\documentclass[reqno,11pt]{amsart}

\usepackage{mynotation}
  \usepackage{psfrag}
\usepackage[pdf,psfrag]{graphviz}

\usepackage[numbers]{natbib}

\newcommand{\e}[1]{\wedge^{{#1}}}

\newcommand{\PGSp}{{\rm PGSp}}

\newcommand{\ch}{{\rm ch}}

\newcommand{\simp}{{\mathrm{sim}}}

\title{The exterior cubic $L$-function of $\GU_6$ and unitary automorphic induction}

\author{Lei Zhang}
\address{Math. Department\\
National University of Singapore\\
Singapore, 119076}
\email{matzhlei@nus.edu.sg}
\date{}



\begin{document}
\begin{abstract}
In this paper, we extend Ginzburg-Rallis' integral representation for the exterior cube automorphic $L$-function of  $\GL_6\times \GL_1$ to that of the quasi-split unitary similitude group  $\GU_6$
and establish its analytic properties to determine the poles of this $L$-function.
Furthermore, we introduce the automorphic induction for automorphic representations of $\GU_n$ and then show that the weak Langlands functorial lift for the automorphic induction exists for generic cuspidal automorphic representations.
By using this automorphic induction, we give a conjectural criterion on the existence of poles of $L(s,\pi,\e{3}\otimes\chi)$ for discrete automorphic representations in the tempered spectrum. 
\end{abstract}

\thanks{The work of the author is supported in part by AcRF Tier 1 grant R-146-000-237-114 and R-146-000-277-114 of National University of Singapore.}

\maketitle
\tableofcontents


\section{Introduction}

Let $F$ be a number field and  $E=F[\sqrt{\tau}]$ be a quadratic extension of $F$,
whose Galois group $\Gal(E/F)$ is generated by a non-trivial automorphism $\iota_{E/F}\colon x\mapsto \bar{x}$. \label{pg:iota-E}
The ring of adeles of $F$ is denoted by $\BA$,
while the ring of adeles of $E$ is denoted by $\BA_E$.
Denote   $J_{2n}$ to  be the skew-symmetric matrix of size $2n\times 2n$
$$
J_{2n}=\begin{pmatrix}
&w_{n}\\ -w_{n}&
\end{pmatrix}, \text{ where $w_n=\begin{pmatrix}
&w_{n-1}\\1&	
\end{pmatrix} $}.
$$
Define the unitary similitude group $\GU_{2n}$ of skew-Hermitian type with respect to $J_{2n}$
\begin{equation}\label{eq:GU-2n}
\GU_{2n}(F)=\{ g\in \Res_{E/F}\GL_{2n}(F)\colon {}^{t}\bar{g}J_{2n}g=\lam(g) J_{2n}\},
\end{equation}
where $\Res_{E/F}\GL_{2n}$ is the Weil restriction of $\GL_{2n}$ from $E$ to $F$ and $\lam(g)\in \BG_{m}$ is the similitude factor of $g$.
Note that $\GU_{2n}$ is quasi-split over $F$.

Denote ${}^L\GU_6=\widehat{\GU}_6(\BC)\rtimes\Gal(E/F)$ to be the $L$-group of $\GU_6$. 
Let $\e{3}$ be the exterior cube representation of $\GL_6(\BC)$.
The representation $\e{3}$ is irreducible and of dimension 20. 
It can be extended to  a representation of $\widehat{\GU}_6(\BC)$.
In Section \ref{sec:exterior}, we will show that 
there are only two irreducible representations of ${}^L\GU_6$ of dimension 20, whose restriction to $\widehat{\GU}_6(\BC)$ is isomorphic to $\e{3}$.
We denote them by $\e{3}$ and $\e{3}\otimes\omega$, where $\omega$ is the non-trivial character of $\Gal(E/F)$.
The main purpose of this paper is to study this automorphic $L$-functions
$L(s,\pi,\e{3}\otimes\chi)$ for  irreducible cuspidal automorphic representations $\pi$ of $\GU_6(\BA)$ in the tempered spectrum, 
where  $\chi$ is a unitary automorphic character of $\GL_1(\BA)$.

To make our computation more understandable, 
let us give a  representation-theoretic explanation for the global integral $\CZ(s,\varphi_\pi,\phi_{\chi_\pi})$,
before we formal definition Ginzburg-Rallis' integral for $\GU_6$ in Section \ref{sec:Global-zeta}. 
 
Let $H$ be the subgroup $G^{\iota_{E/F}}$ of $G$ consisting of elements fixed by the Galois conjugate.
Then  $H$ is isomorphic to the symplectic similitude group $\GSp_6$.
Denote by $P$ the standard Siegel parabolic subgroup of $H$.  
Let $\varphi_\pi$ be  the factorizable vector in $\pi$ and  $\phi_{\chi_\pi}\in \Ind^{\GSp_6(\BA)}_{P(\BA)}\chi_\pi|\cdot|^s$,
where $\chi_\pi$ (given in \eqref{eq:chi-pi}) is a character defined by $\chi$ and the central character $\omega_\pi$ of $\pi$.
The global zeta integral $\CZ(s,\varphi_\pi,\phi_{\chi_\pi})$ is a $\GSp_6(\BA)$-equivariant linear functional in  
\[
\Hom_{\GSp^\triangle_{6}(\BA)}(\pi\otimes\Ind^{\GSp_6(\BA)}_{P(\BA)}\chi_\pi|\cdot|^s,\BC)
\cong
\Hom_{P(\BA)}(\pi\otimes \chi_\pi|\cdot|^s,\BC),	
\]
where $\GSp_6^\triangle$ is the diagonal embedding of $\GSp_6$ into $\GU_6\times \GSp_6$. 
More precisely, for $g\in\GSp_6(\BA)$,
$$
\CZ(s,g*\varphi_\pi,g*\phi_{\chi_\pi})=\CZ(s,\varphi_\pi,\phi_{\chi_\pi})
$$
where $*$ is the right translation.
Here $\Ind^{\GSp_6(\BA)}_{P(\BA)}\chi_\pi|\cdot|^s$   produces an Eisenstein series $E(\phi_\chi,\lam_s)$ on $\GSp_6$, which is involved in the global zeta integral.
 
After the global unfolding process and Fourier expansions, we produce the following $P(\BA)$-equivariant linear functional
in the space
$$
\Hom_{P(\BA)}(\CW(\pi)\otimes \chi_\pi|\cdot|^s,\BC)=\otimes_\nu\Hom_{P(F_\nu)}(\CW(\pi_\nu)\otimes \chi_{\pi,\nu}|\cdot|^s_\nu,\BC),
$$
where $\CW(\pi)$ is the Whittaker model of $\pi$.
Then one obtains an Euler product of local zeta integrals for $\Re(s)$ sufficiently large
$$
\CZ(s,\varphi_\pi,\phi_{\chi_\pi})= \prod_{\nu}\CZ_\nu(s,W_\varphi,\phi_{\chi_\nu}).
$$
It is easy to see that if $\pi$ is not generic, then the global integral $\CZ(s,\cdot)$ is identically zero. 
Thus, we need to establish the analytic properties for $L$-functions for generic cuspidal representations first, 
and then extend those results to the cuspidal representations in the tempered spectrum via the endoscopic classification of $\GU_n$.

Compared with the general linear group case in \cite{GR00},   one of main difficulties in this paper is to evaluate the local zeta integral $\CZ_{\nu}(s,\cdot)$ over unramified places. 
When $\nu$ is split over $E$, $\GU_6(F_\nu)$ is isomorphic to $\GL_6(F_\nu)\times \GL_1(F_\nu)$ and the unramified calculation has been done in \cite{GR00}.
Over the inert places $\nu$, the Casselman-Shalika formula for $\GU_6(F_\nu)$ is given in terms of the characters of $\GSpin_7(\BC)$.
The exterior cube $L$-factor  of $\GU_6(F_\nu)$ can be written as a product of the Spin $L$-factor and the standard $L$-factor of $\GSp_6(F_\nu)$, up to a zeta factor.
Hence the computation of local $L$-factors at inert places is similar to the scenario in \cite{BFG99}. 
We use Pierre's lemma for $\GSpin_7(\BC)$ and representation-theoretic arguments to establish the identity between the local zeta integral $\CZ_\nu(s,W_\varphi,\phi_{\chi_\nu})$ and the local exterior cube $L$-factor.
That is,
the global zeta integral equals
\[
\CZ(s,\varphi_\pi,\phi_{\chi_\pi})= \prod_{\nu\in S}\CZ_\nu(s,W_\varphi,\phi_{\chi_\nu})
\frac{L^S(\frac{s}{2}+\frac{1}{2},\pi\otimes\chi,\e{3})}{\zeta^S(s+2,\omega_\pi\chi^2)\zeta^S(2s+2,\omega_\pi\chi^2)},
\]
where $S$ is a finite set of places of $F$ including all ramified places and archimedean places. 

Next, we follow standard arguments on the analytic properties of the local zeta integrals over the places in $S$, 
and use the location of possible poles of $E(\phi_\chi,\lam_s)$ on $\GSp_6$.
Then we obtain the following theorem.
\begin{thm}\label{thm:pole-intro}
Let $\pi$ be a generic cuspidal automorphic representation of $\GU_6(\BA)$. Then
$L^S(s,\pi\otimes\chi,\e{3})$ is entire unless $\omega_\pi\chi^2$ is a nontrivial quadratic character of $\GL_1(F)\bks \GL_1(\BA)$,
in which case $L^S(s,\pi\otimes\chi,\e{3})$ can have at most a simple pole at $s=0$ or $s=1$.
\end{thm}

 
After we locate the possible poles of $L^S(s,\pi\otimes\chi,\e{3})$, similar to \cite{GR00} and \cite{Y}, a natural question is to ask for a criterion on the existence of the poles of the $L$-functions.
Another difficulty we encounter in this paper is to determine the pole of $L^S(s,\pi\otimes\chi,\e{3})$ by Langlands functoriality 
and then to extend the results to the tempered spectrum.

More precisely, assume that $L^S(s,\pi\otimes\chi,\e{3})$ has a pole.
To determine the poles for the unitary similitude group case, 
we need to consider the Langlands functoriality depending on the quadratic extension $E^+$ associated to $\omega_\pi\chi^2$ in Theorem \ref{thm:pole-intro}.
If $E^+\cong E$, then $\pi$ is an endoscopic lifting of $G(\RU_3\times\RU_3)$, which is a new case compared to $\GL_6$ case in \cite{GR00}.
If $E^+\not\simeq E$, we have a quartic extension $K=E^+\otimes_F E$ over $F$.
In this case, $\pi$ is expected to be an automorphic induction of $\Res_{E^+/F}\GU_3$ defined over a Hermitian space over $K$.


To establish the criterion of the poles and extend Theorem \ref{thm:pole-intro}, we study the endoscopic classification of $\GU_n$ first. 
Following Mok's endoscopic classification of discrete automorphic representations of $\RU_{n}(\BA)$ in \cite{Mk15} 
and
using Xu's framework in \cite{Xu16} and \cite{Xu17},
in Section \ref{sec:endoscopic-GU}, we give an expected conjectural endoscopic classification of discrete automorphic representations of $\GU_{n}(\BA)$ in its tempered spectrum, where automorphic representations share the same Arthur parameters with generic cuspidal automorphic representations.
Remark that the endoscopic classification of $\GU_n$ has been studied in many literatures such as \cite{CHL11}, \cite{M10} and \cite{Shin}.
Based on the endoscopic classification of $\GU_n$, we are able to extend Theorem \ref{thm:pole-intro} to all tempered automorphic representations.

Second, to determine the poles of $L$-functions, let us study the generic vector $v_0$ in
the irreducible representation $\e{3}$ of ${}^L\GU_6$.
That is, $\e{3}({}^L\GU_6)\cdot v_0$ is the open dense orbit in $\e{3}(\BC^{6})$.
Its stabilizer is the following subgroup of ${}^L\GU_6$
\begin{equation}\label{eq:fix-group}
[[((\GL_3\times\GL_3)\rtimes \apair{\eps})\times\GL_1]\rtimes \apair{\iota_{E/F}}]^\circ	
\end{equation}
consisting of all tuples $(g_1,g_2,\eps^j,a,\iota_{E/F}^i)$ satisfying $a\det(g_1)=a\det(g_2)=1$, and $\eps=\ppair{\begin{smallmatrix}
&I_3\\I_3&	
\end{smallmatrix}}$.
Note that  $((\GL_3\times\GL_3)\rtimes \apair{\eps})$ is a subgroup of $\GL_6$.
Therefore, 
for any cuspidal automorphic representation $\pi$ with a generic global Arthur parameter $\wt{\phi}$,
 $L(s,\pi,\e{3}\otimes \chi)$ has a pole at $s=1$ if and only if the image of 
 $(\e{3}\otimes \chi)\circ \wt{\phi}$ is contained in the subgroup \eqref{eq:fix-group}.
 This leads us to the automorphic induction of tempered automorphic representations of $\Res_{E^+/F}\GU_3$.

Hence, we conjecture the following criterion
 \begin{conj}\label{conj:poles-intro}
Suppose that $\pi$ is an irreducible cuspidal automorphic representation of $\GU_6(\BA)$ in the tempered discrete spectrum defined in Conjecture \ref{conj:endoscopy-GU}. 

Then $L^S(s,\pi,\e{3}\otimes\chi)$ has a pole at $s=1$ if and only if $\pi$ is an automorphic induction $i_{E^+/F,\xi}(\tau)$ from an irreducible cuspidal automorphic representation $\tau$ of $\GU^\circ_{K/E^+}(3,\BA_F)$ in the tempered discrete  spectrum and $L(s,i_F(\omega_\tau)\otimes\chi)$ has a pole at $s=1$.
\end{conj}
Here the automorphic induction $i_F(\omega_\tau)$ is an automorphic representation of $\GL_2(\BA)$,
 $K=E\otimes_F E^+$ and $E^+/F$ is the quadratic extension associate to the non-trivial quadratic character $\omega_\pi\chi^2$,  and $\omega_\tau$ is the central character of $\tau$.

To establish Conjecture \ref{conj:poles-intro}, 
one needs to know, for instance,  the criterion of existence of poles of $L(s,\tau_1\otimes\tau_2,\e{2}\otimes {\rm St})$ for $\GL_4\times\GL_2$ (cf. Example \ref{ex}),
where ${\rm St}$ is the standard representation of $\GL_2(\BC)$ and $\e{2}$ is the exterior square of $\GL_4(\BC)$.
We expect that the full proof of Conjecture \ref{conj:poles-intro} is long and complicated, and hence we leave it for the future work.

Analogous to Conjecture in \cite[Section 4]{GR00}, we also conjecture that $L(\frac{1}{2},\pi, \e{3}\otimes\chi)$ is nonzero  if and only if the period integral over certain unitary Ginzburg-Rallis model is not identically zero. 
By using a similar argument in \cite{PWZ} for the general linear case, one can establish that 
if the period integral is nonzero, 
then $L^S(\frac{1}{2},\pi, \e{3}\otimes\chi)\ne 0$,
under certain conditions.
The local multiplicity problems of the analogy of the Ginzburg-Rallis model for the unitary group and the unitary
similitude group cases
have been studied in \cite{WZ18}.

Finally, remark that this exterior cube $L$-function also can be obtained by Langlands-Shahidi method.
For more details, the reader can refer to \cite{KK11}.
However, it is worth to mention that the integral representation of $L$-functions and the period integral of the residual of $L$-functions in Corollary \ref{cor:pole}
have many applications, such as studying the arithmetic properties of special values of automorphic $L$-functions
and
the $\e{3}$-trace formula of $\GU_6(\BA)$  discussed in \cite{JLST} and \cite{A15}, respectively.
In addition, since Conjecture \ref{conj:poles} determines all tempered automorphic representations of $\GU_6(\BA)$, 
one expects to 
match the spectrum side of $\e{3}$-trace formula of $\GU_6(\BA)$ with the spectrum side of the stable trace formula of $\GU^\circ_{K/E^+}(3,\BA_F)$.

The paper is organized as follows. In Section \ref{sec:global}, analogous to the Ginzburg-Rallis' integral, we define the global zeta integral for the $\GU_6$ case and obtain its Euler  product.
In Section \ref{sec:L-factor}, we introduce the exterior cube local  $L$-function of $\GU_6$ and 
compute the local $L$-factor over unramified places. 
Then we study the analytic properties of the partial $L$-function $L^S(s,\pi,\e{3}\otimes\chi)$
and 
give a necessary condition on the existence of poles of the partial $L$-function $L^S(s,\pi,\e{3}\otimes\chi)$ in Section \ref{sec:analytic-L}.
In Section \ref{sec:endoscopic-classificiation}, we give a conjectural endoscopic classification of discrete automorphic representations of $\GU_{n}$ in its tempered spectrum.
In Section \ref{sec:conjecture}, we introduce the automorphic induction for $\GU_n$ and show the existence of poles of cuspidal automorphic representations in tempered discrete spectrum.  
Finally, we state a conjectural criterion on the existence of pole of $L(s,\pi,\e{3}\otimes\chi)$ at $s=1$ in terms of the automorphic induction.

{\it Acknowledgment.} I would like to thank Dihua Jiang for suggesting this problem and for his encouragement to write up this paper. 
Also, I am very grateful to Bin Xu, Chufeng Nien and Fangyang Tian for their comments and valuable suggestions.  

\section{Integral representation}\label{sec:global}

In this section, we introduce the global zeta integral analogous to the  $\GL_6$ case in \cite{GR00} and obtain its Euler product when $\Re(s)$ is sufficiently large. 
In this paper, we focus mainly on the group $\GU_6$, simply denoted by $G$. Take the symplectic subgroup $H$ of $G$ defined by
$$
H=\{g\in \GL_{6}\colon {}^{t}gJ_{6}g=\lam(g) J_{6}\}\cong \GSp_6,
$$
which is the invariant subgroup under the Galois conjugation $g\mapsto \bar{g}$.

\subsection{Global zeta integral}\label{sec:Global-zeta}
Denote $P=MV$ to be the standard Siegel parabolic subgroup of $H$ consisting of elements of form $\ppair{ \begin{smallmatrix}
*&* \\ &*	
\end{smallmatrix} }$.
Its Levi  subgroup $M$ is isomorphic to $\GL_1\times\GL_3$ via  
$$
(a,g)\mapsto\begin{pmatrix}
a g& \\ & g^*	
\end{pmatrix}.
$$
where $g^*=w_3{}^t\!g^{-1}w_3^{-1}$ for $g\in \GL_3$.
Its unipotent radical $V$ consists of elements of form
\begin{equation}\label{eq:v}
v(X):=\begin{pmatrix}
I_{3}& X\\ &I_{3} 	
\end{pmatrix}	
\end{equation}
where $X\in M_{3\times 3}$ satisfies $X=w_3{}^t\!Xw_3^{-1}$.
Let $\delta_P$ be the modular character of $P$.
Then 
$\delta_P(a,g)=|a|^6|\det(g)|^4$,
where $|\cdot|$ is the absolute value over $\BA$.

Let $\chi$ be a character of $\GL_1(F)\bks\GL_1(\BA)$, extended to be a character of $M(\BA)$. We will specify the extension later.
Following the notation of \cite{MW95} (see I.2.17), 
we take an automorphic form 
$$
 \phi_\chi\in \CA(V(\BA)M(F)\bks H(\BA))_{\chi}
$$ 
and form, for $\lam_s\in X_M$,
the Eisenstein series on $H(\BA)$ by
$$
E(\phi_\chi,\lam_s)(g)=\sum_{\gamma\in P(F)\bks H(F)}\lam_s\phi_\chi(\gamma g).
$$
More precisely, we take  for $m=(a,g)\in M(\BA)$ 
$$
 \lam_s(a,g)=|a|^{\frac{3s}{2}}|\det g|^{s}.
$$  
Here we follow the normalization in \cite{Sh10} and use $4s-2$ to replace $s$ in \cite{GR00}.
The theory of Langlands on Eisenstein series shows that $E(\phi_\chi,s)$ converges absolutely for $\Re(s)$ large and has a meromorphic continuation to the complex plane $\BC$.

Let $N_G$ (resp. $N_H$) be the maximal unipotent subgroup of $G$ (resp. $H$) consisting of all upper unipotent matrices.
Fix a nontrivial character $\psi_0$ of $F\bks \BA$ and define a character $\psi_N$ of $N_G(\BA)$ by 
\begin{equation}\label{eq:generic-character}
\psi_N(n)=\psi_0(\tr_{E/F}\frac{1}{2\sqrt{\tau}}(n_{1,2}+n_{2,3})+n_{3,4})	
\end{equation}
where $n=(n_{i,j})\in N_G$.
Let $\pi$ be an irreducible generic cuspidal automorphic representation of $G(\BA)$,
whose central character is denoted by $\omega_\pi$. 
Note that the maximal diagonal torus acts transitively on the set of non-degenerate characters of $N_G(F)\bks N_G(\BA)$. 

Let $\chi$ be a unitary character of $\GL_1(F)\bks \GL_1(\BA)$.
Define $\chi_\pi$ to be a character of $M(\BA)$ given by
\begin{equation}\label{eq:chi-pi}
\chi_\pi(a,g)=(\omega_\pi\chi^3)(a)(\omega_\pi\chi^2)(\det(g)),
\end{equation}
where $\omega_\pi$ is restricted into $\GL_1(F)\bks \GL_1(\BA)$.
Then for $z\in Z(\BA)$
$$
E(\phi_{\chi_\pi},\lam_s)(zg)=\omega_\pi(z)^{-1}E(\phi_{\chi_\pi},\lam_s)(g).
$$
where $Z$ (or $Z_H$) is the center of $H$. 
The Eisenstein series is normalized as the following
$$
E^*(\phi_{\chi_\pi},\lam_s)(g):=L^S(s+2,\omega_\pi\chi^2)L^S(2s+2,\omega^2_\pi\chi^4)E(\phi_{\chi_\pi},\lam_s)(g).
$$
where $S$ is a finite set of  all archimedean places and ramified places such that outside $S$ all data are unramified,
and $L^S(s,\cdot)=\prod_{\nu\notin S}L_\nu(s,\cdot)$ is the partial $L$-function.
Referring to \cite{G95} for instance, the poles of the normalized Eisenstein are given as follows.
\begin{lm}[Lemma 5.4 \cite{G95}]\label{lm:pole-Eisenstein}
For $\Re(s)\geq 0$,
\begin{enumerate}
 	\item if $\omega^2_\pi\chi^4\ne 1$, then $E^*(\phi_{\chi_\pi},\lam_s)$ is entire;
 	\item if $\omega^2_\pi\chi^4=1$ and $\omega_\pi\chi^2\ne 1$, then $E^*(\phi_{\chi_\pi},\lam_s)$ has a simple pole at $s=1$;
 	\item if $\omega_\pi\chi^2=1$, then $E^*(\phi_{\chi_\pi},\lam_s)$ has a simple pole at $s=1$ and $s=2$.
 \end{enumerate} 
\end{lm}

Let $\varphi_\pi$ be a cuspidal automorphic form in $\pi$. 
Write $\iota$ to be the involution of $G$ given by
$$
\iota(g)= \begin{pmatrix}
I_2&&\\ &w_2& \\ &&I_2	
\end{pmatrix} g \begin{pmatrix}
I_2&&\\ &w_2& \\ &&I_2	
\end{pmatrix}^{-1}.
$$ 
For a function $\varphi$ on $G$, denote $\varphi^\iota=\varphi\circ\iota$.

The {\it global zeta integral} $\CZ(s,\varphi_\pi,\phi_{\chi_\pi})$ is defined by
\begin{equation}
\CZ(s,\varphi_\pi,\phi_{\chi_\pi})=\int_{Z(\BA)H(F)\bks H(\BA)}\varphi^\iota_\pi(g)E(\phi_{\chi_\pi}\lam_s)(g)\ud g,
\end{equation}
and is simplified as $\CZ(\cdot)$ if no confusion is caused.
In addition, we may use the partial $L$-functions $L^S(s+2,\omega_\pi\chi^2)L^S(2s+2,\omega^2_\pi\chi^4)$ to  normalize $\CZ(s,\varphi_\pi,\phi_{\chi_\pi})$, that is,
\begin{equation}\label{eq:normalized-zeta-section-2}
\CZ^*(s,\varphi_\pi,\phi_{\chi_\pi})=L^S(s+2,\omega_\pi\chi^2)L^S(2s+2,\omega^2_\pi\chi^4)\CZ(s,\varphi_\pi,\phi_{\chi_\pi}).
\end{equation}

Because $\varphi_\pi$ is cuspidal and $E(\phi_{\chi_\pi},\lam_s)$ is of moderate growth, the global zeta integral $\CZ(s,\varphi_\pi,\phi_{\chi_\pi})$ converges absolutely and is meromorphic in $s\in \BC$.
The functional equation for $\CZ(s,\varphi_\pi,\phi_{\chi_\pi})$ relating $s$ to $-s$ follows from the functional equation of the normalized Eisenstein series $E(\phi_{\chi_\pi},\lam_s)$.

\subsection{The eulerian property of the global integral}\label{sec:euler}
For $\Re(s)$ sufficient large, we are able to unfold the Eisenstein series
\begin{align*}
\CZ(\cdot)= &\int_{Z(\BA)P(F)\bks H(\BA)}\varphi^\iota_\pi(g)\phi_\chi(g)\ud g\\
=&\int_{Z(\BA)M(F)V(\BA)\bks H(\BA)}\phi_\chi(g)\int_{[V]}\varphi^\iota_\pi(vg)\ud v\ud g.
\end{align*}
Here for an algebraic group $X$, as a convention $[X]$ denotes the quotient $X(F)\bks X(\BA)$.

Write $V_G=\{v(X)\mid w_3{}^t\!\bar{X}w_3=X,~X\in \Res_{E/F}M_{3\times 3}\}$, the unipotent subgroup of $G$. It can be decomposed as
\begin{equation}\label{eq:V-G}
V_G=V\oplus V_G^{-}\text{ with }
V_G^-=\{v(\sqrt{\tau}A)\mid w_3{}^t\!Aw_3=-A,~A\in M_{3\times 3}\}.
\end{equation}
Applying the partial Fourier expansion of $\varphi_\pi$ on $[V^-_G]$,
$\varphi^\iota_\pi$ is equal to
\begin{equation}\label{eq:F-1}
\sum_{\{\alpha \in M_{3\times 3}(F)\mid w_3{}^t\!\alpha w_3=-\alpha \}}\int_{[V_{G}^-]} \varphi^\iota_\pi(v(\sqrt{\tau} A)vg)\psi_0(\frac{\tr\alpha A}{2})\ud A.	
\end{equation}

Let $M(F)$ act on the character group of $[V^-_G]$, 
by the action $(a,g)\cdot A=aw_3{}^t\!gw_3Ag$, and then it has two orbits.
We choose representatives $v(0)=I_6$ and $v(\alpha_1)$ where 
\begin{equation}\label{eq:alpha-1}
\alpha_1= \begin{pmatrix}
0&1&0 \\ 0&0&-1 \\ 0&0&0	
\end{pmatrix},
\end{equation}
and their stabilizers are $M(F)$ and $L_1(F) V^+_1(F)$ respectively, where
$$
L_1=\cpair{
\diag(\det(g),g,g^*,1)\mid g\in \GL_2}
$$
and $V^+_1$ consists of elements of form
\begin{equation}\label{eq:v-1}
v_{1}(x,y):=
\diag\ppair{\begin{pmatrix}
1&x&y\\ &1&0 \\ &&1	
\end{pmatrix}, \begin{pmatrix}
1&0&-y\\ &1&-x\\ &&1	
\end{pmatrix}}.	
\end{equation}
After changing variable, by the automorphy of $\varphi_\pi$, we rewrite \eqref{eq:F-1} as
$$
\int_{A} \varphi^\iota_\pi(v(\sqrt{\tau} A)vg)\ud A 
+ \sum_{\gamma\in  L_{1}V^+_1(F)\bks M(F)}\int_{A} \varphi^\iota_\pi(v(\sqrt{\tau} A)\gamma vg)\psi_0(\frac{\tr\alpha A}{2})\ud A.
$$

Similar to Page 248 in \cite{GR00}, this Fourier expansion is absolutely convergent for $\Re(s)$ sufficient large, due to the estimation on the cuspidal form $\varphi_\pi$ in \cite{MW95}.
We may replace $\varphi_\pi$ by its Fourier expansion, combine the integrals over $[V]$ and $[V_G^{-}]$, and then obtain
that $\CZ(\cdot)$ equals 
\begin{align}
&\int_{Z(\BA)M(F)V(\BA)\bks H(\BA)}\phi_{\chi}(g)\int_{[V_G]} \varphi^\iota_\pi(vg)\ud v\ud g \label{eq:F-1-cusp}\\
+&\int_{Z(\BA)L_1V^+_1(F)V(\BA)\bks H(\BA)}\phi_{\chi}(g)
\int_{[V_G]} \varphi^\iota_\pi(v(X)g)\psi_0(\tr_{E/F}\frac{\tr\alpha X}{4\sqrt{\tau}})\ud v(X)\ud g. \label{eq:F-2}   
\end{align}
Due to the cuspidality of $\varphi_\pi$, the term \eqref{eq:F-1-cusp} equals 0. 
Then $\CZ(\cdot)$ equals the term \eqref{eq:F-2}.

Next, we will perform the above {\it Fourier-expansion process} several times.
For simplicity, we will skip repeated details.
Set $V^{-}_{1}$ to be the subgroup consisting of elements $v_1(\sqrt{\tau}x,\sqrt{\tau}y)$ for $x,y\in\BG_a$. 
Continue \eqref{eq:F-2}, and consider the Fourier expansion on $[V^-_1]$.
The action of $L_1(F)$ on the character group of $[V^-_1]$ has two orbits: the trivial character and the set of non-trivial characters. 
After combing the inner integrals $\int_{[V^+_1]}\int_{V_G}$, the Fourier coefficient associated to the trivial character vanishes due to the cuspidality of $\varphi_\pi$. We choose the non-trivial character on $[V^-_1]$ defined by
$v_1(\sqrt{\tau}x,\sqrt{\tau}y)\mapsto \psi_0(x)$.
Then its stabilizer in $L_1(F)$ is $L_2V^+_2(F)$, where
$$
L_2=\{\diag(a,a,1,a,1,1)\mid a\in \BG_m\}
$$
and $V^+_2$ consists of elements of form
\begin{equation}\label{eq:v-2}
v_2(x):=\diag\ppair{1,\begin{pmatrix}
1&x\\ &1	
\end{pmatrix},\begin{pmatrix}
1&-x\\ &1	
\end{pmatrix},1}.
\end{equation}

Following a similar Fourier-expansion process, we have $\CZ(\cdot)$ equal to
\begin{align}
&\int_{\BZ(\BA)(L_2V_2^+)(F)(V^+_1V)(\BA)\bks H(\BA)}\phi_\chi(g)
\nonumber \\
&\times \int_{[V_1V_G]}\varphi^\iota_\pi(vg)\psi_0(\tr_{E/F}\frac{v_{1,2}+v_{2,4}}{2\sqrt{\tau}})\ud v\ud g,\label{eq:F-3}
\end{align}
where $V_1=V_1^+\oplus V^-_1$ and $v=(v_{i,j})\in V_1V_G$.

Next, we consider the Fourier expansion on $V^-_2$, which consists of elements $v_2(\sqrt{\tau}x)$ for $x\in \BG_a$.
Let us combine the integrals over unipotent subgroups,
\begin{equation}\label{eq:F-4-1}
\int_{[V^+_2V_1V_G]}\sum_{\alpha\in F}\int_{[V^-_2]}
\varphi^\iota_\pi(v_2(\sqrt{\tau}x)vg) \psi_0(\tr_{E/F}\frac{v_{1,2}+v_{2,4}}{2\sqrt{\tau}}+\alpha x)\ud x\ud v. 
\end{equation}
Now we apply the technique of `exchanging roots' (see Chapter 7 \cite{GRS11} for instance). 
Because $\varphi_\pi$ is left $G(F)$-invariant, 
for $\alpha\in F$,
\begin{align}
&\varphi^\iota_\pi(v_2(\sqrt{\tau}x)vg)=\varphi^\iota_\pi(v_3(-\alpha)v_2(\sqrt{\tau}x)vg)\nonumber \\    
=&\varphi^\iota_\pi(v_2(\sqrt{\tau}x)v_3(-\alpha)v\begin{pmatrix}
0&0&0\\ \sqrt{\tau}\alpha x&\tau \alpha x^2&0\\ 0&-\sqrt{\tau}\alpha x&0	
\end{pmatrix}vg), \label{eq:F-4-2}
\end{align}
where $v_3(\alpha)=\diag(I_2,\ppair{\begin{smallmatrix}
1& \alpha\\ &1	
\end{smallmatrix}},I_2)$. 
Plugging \eqref{eq:F-4-2} into \eqref{eq:F-4-1} and changing variable $v\mapsto v(\cdot)^{-1}v$, we have
\begin{equation}
\int_{[V^+_2V_1V_G]}\sum_{\alpha\in F}\int_{[V^-_2]}
\varphi^\iota_\pi(v_2(\sqrt{\tau}x)v_3(-\alpha)vg) \psi_0(\tr_{E/F}\frac{v_{1,2}+v_{2,4}}{2\sqrt{\tau}})\ud x\ud v. 
\label{eq:F-4-3}	
\end{equation}

Set $V_3=\{v_3(x)\mid x\in \BG_a\}$ and $V^c_3=\{n=(n_{i,j})\in N_G\mid n_{3,4}=0\}$. Note that $N_G=V_3V^c_3$.
Continuing \eqref{eq:F-4-3}, we have
$$
\sum_{\alpha\in F}\int_{[V_3]}\int_{[V^c_3]}\varphi^\iota_{\pi}(vv_3(x-\alpha)g)\psi_0(\tr_{E/F}\frac{v_{1,2}+v_{2,4}}{2\sqrt{\tau}})
\ud v\ud x
$$
$$
=\int_{V_3(\BA)}\int_{[V^c_3]}\varphi^\iota_{\pi}(vv_3(x)g)\psi_0(\tr_{E/F}\frac{v_{1,2}+v_{2,4}}{2\sqrt{\tau}})
\ud v\ud x.
$$

Take the Fourier expansion on $[\iota(V_3)]$.
Similarly, $L_2(F)$ acts on the character group of $[\iota(V_3)]$ with two orbits. 
The trivial character eventually contributes $0$ to the global zeta integral, due to the cuspidality of $\varphi_\pi$. 
The stabilizer of $L_2(F)$ on non-trivial characters is identity. 
Choosing the character $\iota(v_3(x))\mapsto \psi_0(x)$ of $[\iota(V_3)]$, by $\iota(V^c_3\cdot \iota(V_3))=N_G$, we obtain that $\CZ(\cdot)$ equals
\begin{align*}
&\int_{Z(\BA)N_H(\BA)\bks H(\BA)}\phi_\chi(g) 
\int_{V_3(\BA)}\int_{[N_G]}\varphi_\pi(n\cdot\iota(v_3(x)g))
\psi_N(n)\ud n\ud x\ud g \nonumber\\
=&\int_{Z(\BA)N_H(\BA)\bks H(\BA)}\phi_\chi(g) 
\int_{V_3(\BA)}W_\varphi^\iota(v_3(x)g)\ud x\ud g,  
\end{align*} 
where $W_\varphi(g)$ is the Whittaker function of $\varphi_\pi$ with respect to $\psi_N$.

Finally, we may define the local zeta integral by
\begin{align}
 & \CZ_\nu(s,W_\varphi,\phi_{\chi_\nu})\nonumber \\
=& \int_{ Z(F_\nu)N_H(F_\nu)\bks H(F_\nu)}\phi_{\chi_\nu}(g_\nu) 
\int_{V_3(F_\nu)}W_{\varphi_\nu}(\iota(v_3(x_\nu)g_\nu))\ud x_\nu\ud g_\nu.	\label{eq:zeta-eulerian}
\end{align}
Remark that  the inner integral considered as a function of the variable $g_\nu$ is left $N_H(F_\nu)$-invariant. 
That is  because the restriction $\psi_N$ (defined in \eqref{eq:generic-character}) into $v_1(x,y)$ (in \eqref{eq:v-1}) and $v_2(x)$ (in \eqref{eq:v-2}) are trivial.

The calculations discussed previously can be summarized as follows.
\begin{pro}\label{pro:Eulerian}
Let $\pi$ be an irreducible generic cuspidal automorphic representation of $G(\BA)$ and $\varphi_\pi$ be a cuspidal automorphic form of $\pi$. 
Assume that $\varphi_\pi$ is factorizable. 

Then for $\Re(s)$ sufficient large, the global zeta integral $\CZ(s,\varphi_\pi,\phi_{\chi_\pi})$ is eulerian, that is,
$$
\CZ(s,\varphi_\pi,\phi_{\chi_\pi})=\prod_{\nu}\CZ_\nu(s,W_\varphi,\phi_{\chi_\nu})
$$
where the product is taken over all local places $\nu$ of $F$.
\end{pro}

\section{Unramified Local $L$-functions} \label{sec:L-factor}
In this section, we start with recalling the definition of the $L$-group ${}^L\GU_{n}$ and explicate the exterior cube representation of ${}^L\GU_{6}$.
Then we will explicitly write down the local exterior cube $L$-functions of the unramified representations. 
After that, we will compute the local zeta integral $\CZ_\nu(\cdot)$ over all unramified places and obtain the desired exterior cube $L$-function.

\subsection{The twisted exterior cube representation} \label{sec:exterior}
Referring to \cite{KK08} for instance, recall that the $L$-group ${}^L\GU_{n}$ of $\GU_{2n}$ is isomorphic to $( \GL_{2n}(\BC)\times\GL_{1}(\BC))\rtimes \apair{\sig}$, where
$$
\sig(g,a)=(J_{2n}{}^{t}g^{-1}J_{2n},a\det(g)),
$$
where we write $\sig$ to $\iota_{E/F}$  in $\Gal(E/F)$ for short.
It is well known that
all irreducible representations of $\GL_{2n}(\BC)$ are parameterized by partitions $(r_1,r_2,\dots,r_{2n})$ where $r_1\geq r_2\geq \cdots\geq r_{2n}$. 
For instance, the $n$-th  exterior power   $\wedge^n$ of $\GL_{2n}(\BC)$ corresponds to the partition $(\underbrace{1,\dots,1}_{n},\underbrace{0,\dots,0}_{n})$.

For an irreducible representation $\rho$ of $\GL_{2n}(\BC)$,
denote $\rho_{m}$ to be an irreducible representation of $\GL_{2n}(\BC)\times\GL_{1}(\BC)$ extended by
$$
\rho_{m}\colon (g,a)\mapsto \rho(g)a^{m}. 
$$
We generalize  Lemma 2.1 in \cite{KK08} to the following lemma.
\begin{lm}\label{lm:extension-cond}
The irreducible representation $\rho_{m}$  is extendable to ${}^{L}\GU_{2n}$ if and only if $r_{i}+r_{2n+1-i}=m$ for all $i$. 
\end{lm}
\begin{proof}
It is easy to see that the partition associated to $\rho_m\circ\sig$ is $(m-r_{2n},m-r_{2n-1},\dots,m-r_{1})$.
The irreducible representation $\rho_{m}$  is extendable to ${}^{L}\GU_{2n}$ if and only if $\rho_{m}$ and $\rho_{m}\circ\sig$ are isomorphic,
in which case  $r_{i}+r_{2n+1-i}=m$ for all $i$. 
In such case, there is an intertwining operator $A$ such that $\rho_{m}\circ\sig =A \rho_{m}A^{-1}$, where $A\in \GL_{\dim(\rho)}(\BC)$.
The representation $\rho_{m}$ extends on ${}^{L}\GU_{2n}$ by
$$
(g,a,1)\mapsto a^m\rho(g)  \text{ and }
(1,1,\sig)\mapsto A.	
$$
\end{proof}
Note that in general  the extension of $\rho_m$ is not unique.

Following the proof of Lemma \ref{lm:extension-cond}, $\wedge^n$ can be extended to   an irreducible representation of ${}^{L}\GU_{2n}$. 
First we extend $\wedge^n$ to   a representation of $\GL_{2n}(\BC)\times \GL_1(\BC)$ via $(g,a)\mapsto a\wedge^n(g)$, still denoted by $\wedge^n$ for simplicity.
Let $\rho_1$ be an extension of $\wedge^n$ to ${}^{L}\GU_{2n}$. 
We will give an explicit description of $\rho_1$.

Let $\cpair{e_{1},e_{2},\dots,e_{2n}}$ be the standard basis for $\BC^{2n}$.
Then 
\begin{equation}\label{eq:basis}
\CB_n=\cpair{e_{i_1}\wedge e_{i_2}\wedge\cdots\wedge e_{i_{n}}\colon 1\leq i_1<i_2<\cdots<i_n\leq 2n}	
\end{equation}
is a basis for $\wedge^{n}(\BC^{2n})$.
For $v=e_{i_1}\wedge e_{i_2}\wedge\cdots\wedge e_{i_{n}}$ and 
$w=e_{j_1}\wedge e_{j_2}\wedge\cdots\wedge e_{j_{n}}$ in $\CB_n$,
define the bilinear form $q$ by
\begin{equation}\label{eq:q}
q(v,w)e_{1}\wedge e_{2}\wedge\cdots\wedge e_{2n}=v\wedge w.	
\end{equation}
We have $q(v,w)=(-1)^nq(w,v)=\pm 1$, and for $g\in\GL_{2n}(\BC)$
\begin{equation}\label{eq:L-GU6-q}
q(ge_{i_1}\wedge\cdots\wedge ge_{i_{n}},ge_{j_1}\wedge\cdots\wedge ge_{j_{n}})
=\det(g)q(v,w).
\end{equation}
Thus, $q(\cdot,\cdot)$ is a symplectic (resp. symmetric) form of $\wedge^{n}(\BC^{2n})$ if $n$ is odd (resp. even).
 
Write $N=\dim \wedge^{n}(\BC^{2n})$ in the rest of  this section.
Let $S$ be the linear map associated to $q$ defined by, for all $i_1<i_2<\cdots<i_n\leq 2n$,
$$
S\colon e_{i_1}\wedge e_{i_2}\wedge\cdots\wedge e_{i_{n}}\mapsto q(e_{i^{-}_1}\wedge\cdots\wedge e_{i^{-}_{n}},e_{i_1}\wedge\cdots\wedge e_{i_{n}})e_{i^{-}_1}\wedge\cdots\wedge e_{i^{-}_{n}},
$$
where $\cpair{i^{-}_1,i^{-}_2,\dots,i^{-}_{n}}\cup \cpair{i_1,i_2,\dots,i_{n}}$ is a partition of $\{1,2,\dots,2n\}$ and $i^{-}_1<i^{-}_2<\dots<i^{-}_{n}$. 
Then $S^{2}=(-1)^nI_{N}$ and $S=(-1)^n\cdot {}^{t}\!S$.
By \eqref{eq:L-GU6-q}, we have
\begin{equation}\label{eq:wedge}
{}^t\!\wedge^n(g)\cdot S\cdot\wedge^n(g)=\det(g)S. 	
\end{equation}
Hence, if $n$ is odd, $\wedge^n(g)$ is in $\GSp_{N}(S,\BC)$;
if $n$ is even, $\wedge^n(g)$ is in $\GSO_N(S,\BC)$, where
$$\GSO_N(S,\BC)=\{U\in\GL_{N}(\BC)\colon {}^t U\cdot S\cdot U=\lam(U) S, \det(U)=\lam(U)^{N/2}\}$$
and $\lam(\cdot)$ is the similitude character of $\GSO_N(S,\BC)$.

Write $A$ to be the image $\rho_1((1,1,\sig))$ in $\GL_{N}(\BC)$.
Then it satisfies $A^2=I_{N}$. Following the proof of Lemma \ref{lm:extension-cond},  for all $g\in \GL_{2n}(\BC)$, we have
\begin{equation}\label{eq:L-GU6-rho}
\det(g)\cdot \wedge^n(J_{2n} {}^{t}\!g^{-1}J^{-1}_{2n})=A\cdot \wedge^n(g)\cdot A^{-1}
\end{equation}
Plugging \eqref{eq:wedge} into \eqref{eq:L-GU6-rho}, one has that $A^{-1}\cdot S\cdot \e{n}(J_{2n})$ is commutative with the image $\e{n}(\GL_{2n}(\BC))$.
Since $\e{n}$ is irreducible, by Schur Lemma, $A^{-1}\cdot S\cdot \e{n}(J_{2n})$ is in the center of $\GL_{N}(\BC)$.
We may assume that $A=zS\cdot \e{n}(J_{2n})$ for some $z\in \BC^\times$.
It is easy to check that $\e{n}(J_{2n})^{2}=(-1)^n\cdot I_{N}$ and $\e{n}(J_{2n})=(-1)^n\cdot{}^{t}\!\e{n}(J_{2n})$.
Since $\e{n}(J_{2n})$ satisfies \eqref{eq:wedge}, we have $S\cdot \e{n}(J_{2n})=\e{n}(J_{2n})\cdot S$.
Then $A=\pm S\cdot \e{n}(J_{2n})$ as $A^2=I_{N}$.

Let $\omega$ be the character of ${}^{L}\GU_{2n}$ given by 
$$
(g,a,1)\mapsto 1 \text{ and  }
(1,1,\sig)\mapsto -1.
$$
In particular, denote $\e{n}$ to be the extension of $\rho_1$ to an irreducible representation of ${}^L\GU_{2n}$ such that
\begin{equation}\label{eq:rho}
\e{n}\colon (g,a,1)\mapsto a\e{n}(g)\text{ and }
(1,1,\sig)\mapsto S\cdot \e{n}(J_{2n}).
\end{equation}
Write $\e{n}\otimes\omega$ for the other extension different by
$\e{n}\otimes\omega\colon (1,1,\sig)\mapsto -S\cdot \e{n}(J_{2n})$.
Then we have the following
\begin{lm}\label{lm:exterior-cube-repn}
The extensions $\e{n}$ and $\e{n}\otimes\omega$ are the only two non-isomorphic irreducible representations of ${}^{L}\GU_{2n}$.
\end{lm}
\begin{proof}
Let $\rho_1$ be an extension of $\e{n}$.
Following the above discussion, there are only two choices for the image $\rho_1((1,1,\sig))$.
Hence it suffices to show that $\e{n}$ and $\e{n}\otimes\omega$ are not isomorphic. 
We only need to show that $S\cdot \e{n}(J_{2n})$ and $-S\cdot \e{n}(J_{2n})$ are not conjugate in $\GL_{N}(\BC)$. 
Write $A=S\cdot \e{n}(J_{2n})$. 
Then $A={}^tA$ and $A^2=I_N$, which implies that $A$ is diagonalizable and has only eigenvalues $\pm 1$.

We will show that $\dim\ker(A-I_N)\ne \dim\ker(A+I_N)$, where $\dim\ker(A\pm I_N)$ is the dimension of the eigenspaces with respect to $\pm 1$.
Denote $\CB^\pm=\{v\in\CB_n\colon A(v)=\pm v\}$, where $\CB_n$ is the basis \eqref{eq:basis} and $\CB^c=\CB_n\smallsetminus (\CB^+\cup\CB^-)$.
Remark that $\CB^+\cup \CB^-$ is not empty for all $n\geq 1$.
Under the action of $A$, we have the decomposition 
$$
\e{n}(\BC^{2n})={\rm Span}_{\BC}\CB^c\oplus {\rm Span}_{\BC}\CB^+\oplus {\rm Span}_{\BC}\CB^-.
$$ 
Restricted to the subspace ${\rm Span}_{\BC}\CB^c$, due to $A^2=I_N$, $A$ has the same dimensions of the eigenspaces with respect to $\pm 1$.
It suffices to show that $\#\CB^+\ne \#\CB^-$, that is, $\CB^\pm$ have the different cardinality. 

Let $v=e_{i_1}\wedge e_{i_2}\wedge\cdots\wedge e_{i_n}\in \CB^+\cup\CB^-$, where $1\leq i_1<i_2<\cdots<i_n\leq 2n$. 
Note that by tracking the index $i_1<i_2<\cdots<i_n$ we have $v\in\CB^+\cup\CB^-$ if and only if $\{i_1,i_2,\dots,i_n\}\cup \{2n+1-i_1,2n+1-i_2,\dots,2n+1-i_{2n}\}$ is a partition of $\{1,2,\dots,2n\}$.
Then 
\begin{equation}\label{eq:e-J}
\e{n}(J_{2n})(v)=(-1)^\alpha e_{2n+1-i_1}\wedge e_{2n+1-i_2}\wedge\cdots\wedge e_{2n+1-i_n},	
\end{equation}
where $\cup^{n}_{k=1}\{i_{k},2n+1-i_k\}$ is a partition of $\{1,2,\dots,2n\}$ and $\alpha=\#\{i_k\colon i_k\leq n\}=\#\{i_k\colon i_k<2n+1-i_{k}\}$.

To compute $S(e_{2n+1-i_1}\wedge e_{2n+1-i_2}\wedge\cdots\wedge e_{2n+1-i_n})$,
we need to calculate $q(v,e_{2n+1-i_n}\wedge e_{2n+1-i_{n-1}}\wedge\cdots\wedge e_{2n+1-i_1})$ by the definition of $S$.
First, we have 
\begin{equation}\label{eq:q-v-1}
 v\wedge e_{2n+1-i_n}\wedge e_{2n+1-i_{n-1}}\wedge\cdots\wedge e_{2n+1-i_1} 
=(-1)^\beta e_{i'_1}\wedge \cdots\wedge e_{i'_n}\wedge e_{i'_{n+1}}\wedge \cdots\wedge e_{i'_{2n}}	
\end{equation}
where $\beta=\#\{i_{k}\colon i_k>2n+1-i_k\}=n-\alpha$ and $i'_{k}=\min\{i_k,2n+1-i_k\}$ for $1\leq k\leq n$.
Then $(i'_1~i'_2~\cdots~i'_n)$ and $(i'_{n+1}~i'_{n+2}~\cdots~i'_{2n})$  are permutations of $\{1,2,\dots,n\}$ and $\{n+1,n+2,\dots,2n\}$ respectively.
Due to the symmetry $i'_{k}+i'_{2n+1-k}=2n+1$,  
\begin{equation}\label{eq:q-v-2}
e_{i'_1}\wedge \cdots\wedge e_{i'_n}\wedge e_{i'_{n+1}}\wedge \cdots\wedge e_{i'_{2n}}=e_1\wedge e_2\wedge \cdots\wedge e_{2n}.	
\end{equation}
Combining \eqref{eq:q-v-1} and \eqref{eq:q-v-2}, we have 
$$q(v,e_{2n+1-i_n}\wedge e_{2n+1-i_{n-1}}\wedge\cdots\wedge e_{2n+1-i_1})=(-1)^{n-\alpha}.$$
By 
$$e_{2n+1-i_1} \wedge\cdots\wedge e_{2n+1-i_n}=(-1)^{[\frac{n}{2}]}e_{2n+1-i_n}\wedge e_{2n+1-i_{n-1}}\wedge\cdots\wedge e_{2n+1-i_1},$$ 
then
$$
S\colon e_{2n+1-i_1}\wedge e_{2n+1-i_2}\wedge\cdots\wedge e_{2n+1-i_n}\mapsto (-1)^{n-\alpha+[\frac{n}{2}]}e_1\wedge e_2\wedge\cdots\wedge e_n.
$$

Composing with \eqref{eq:e-J}, one has 
\begin{equation}
A(v)=(-1)^{n+[\frac{n}{2}]}v \text{ for all $v\in\CB_n\setminus \CB^c$.}
\end{equation}
Therefore, one of the sets $\CB^+$ and $\CB^-$ is empty and $\#\CB^+\ne \#\CB^-$ as $\CB^+\cup\CB^-\ne \emptyset$.
It implies that $A$ and $-A$ are not conjugate in $\GL_N(\BC)$, and then $\e{n}$ and $\e{n}\otimes\omega$ are not isomorphic. 
\end{proof}

In order to define the local exterior cube $L$-factor over unramified places, let us give the details on the case $n=3$ based on the proof of Lemma \ref{lm:exterior-cube-repn}.
Write $A=S\e{3}J_6$ and we have
\begin{equation}\label{eq:A-1}
A(e_{i}\wedge e_{j}\wedge e_{k})= 
e_{i}\wedge e_{j}\wedge e_{k} \text{ if }\cpair{i,j,k}\cap\cpair{7-i,7-j,7-k}=\emptyset.
\end{equation}
Also for the remaining 12 vectors in the basis $\CB_3$, one has 
\begin{equation}\label{eq:A-2}
\begin{array}{clcl}
A\colon &e_{1}\wedge e_{2}\wedge e_{6}\mapsto e_{2}\wedge e_{3}\wedge e_{4}&& 
e_{1}\wedge e_{3}\wedge e_{6}\mapsto -e_{2}\wedge e_{3}\wedge e_{5}\\
&e_{1}\wedge e_{4}\wedge e_{6}\mapsto -e_{2}\wedge e_{4}\wedge e_{5}&& 
e_{1}\wedge e_{5}\wedge e_{6}\mapsto  e_{3}\wedge e_{4}\wedge e_{5}\\
&e_{1}\wedge e_{2}\wedge e_{5}\mapsto -e_{1}\wedge e_{3}\wedge e_{4}&& 
e_{2}\wedge e_{5}\wedge e_{6}\mapsto -e_{3}\wedge e_{4}\wedge e_{6}.
\end{array}
\end{equation}

Therefore, by \eqref{eq:A-1} and \eqref{eq:A-2}, $\e{3}((\diag\cpair{x,y,z,1,1,1},a,\sig))$ is conjuage to 
\begin{align}
a\cdot & \diag\left\{xyz,xy,xz,yz,x,y,z,1,
\begin{pmatrix}
&xy\\yz &
\end{pmatrix},~\begin{pmatrix}
&xz\\yz &
\end{pmatrix}\right.\nonumber  \\
&\left. 
~\begin{pmatrix}
&xy\\xz &
\end{pmatrix},
~\begin{pmatrix}
&x\\y &
\end{pmatrix},~\begin{pmatrix}
&x\\z &
\end{pmatrix}
,~\begin{pmatrix}
&y\\z &
\end{pmatrix}
 \right\}    
\end{align}
Remark that   the format of the above semisimple element in $\GL(\e{3}\BC^6)$ does not affect the $L$-factor of unramified representations.

\subsection{Local $L$-functions}\label{sec:local-L-factor}
Let $\nu$ be a local place of $F$. Fix a uniformizer $\varpi_\nu$ of $F_\nu$.
If $\nu$ is inert, then $E_\nu$ is the unramified quadratic extension of $F_\nu$
and $\varpi_\nu$ is also a uniformizer of $E_\nu$.
If $\nu$ is split, then $\nu$ split as $\nu_1$ and $\nu_2$ and $E_\nu=F_{\nu}\times F_{\nu}$.
 
Assume that $\nu$ is inert.  
Let $B=TU$ be the Borel subgroup of $\GU_{6}$ consisting of all upper triangular matrices, whose Levi subgroup $T$ consists of elements
$$
t:=\diag\cpair{t_{1},t_2,t_{3},a\bar{t}_{3}^{-1},a\bar{t}_{2}^{-1},a\bar{t}_{1}^{-1}}.
$$
Let $\chi_{i}$ for $1\leq i\leq 3$ be  characters of $E_{\nu}^{\times}$ and $\chi_{0}$ be a character of $F_{\nu}^{\times}$.
Then one has a character
$\chi_{1}\otimes\chi_{2}\otimes\chi_{3}\otimes\chi_{0}$ of $T(F_\nu)$ defined by $t\mapsto \chi_{1}(t_1)\chi_{2}(t_2)\chi_{3}(t_3)\chi_{0}(a)$.
Let $I(\chi_{1}\otimes\chi_{2}\otimes\chi_{3}\otimes\chi_{0})$ be the normalized induced representation of $\GU_{6}(F_{\nu})$. 

Over unramified places $\nu$, assume that $\pi_\nu$ is the  unramified constituent of $I(\chi_{1}\otimes\chi_{2}\otimes\chi_{3}\otimes\chi_{0})$ for unramified characters $\chi_i$.
Referring to \cite{KK08} for instance, we have that the Satake parameter of $\pi_\nu$ is the semi-simple conjugacy class of
\begin{equation}\label{eq:Satake}
t_{\pi_\nu}:=(\diag\cpair{\chi_{1}(\varpi_\nu),\chi_{2}(\varpi_\nu),\chi_{3}(\varpi_\nu),1,1,1},\chi_{0}(\varpi_\nu),\sig)	
\end{equation}
in $(\GL_{6}(\BC)\times\GL_{1}(\BC))\rtimes\apair{\sig}$.
The central character of $\omega_{\pi_\nu}$ is given by 
\begin{equation}\label{eq:center-character}
\omega_{\pi_\nu} =\chi_1 \chi_2 \chi_3\cdot  \chi_0\circ N_{E_\nu/F_\nu}.
\end{equation}

Suppose that $\nu$ is split. Referring to Section 2 \cite{KK08}, $\GU_6(F_\nu)$ is isomorphic to $\GL_6(F_\nu)\times \GL_{1}(F_\nu)$.
Let $I(\otimes^6_{i=1}\chi_i)\otimes\chi_0$ be the normalized representation of $\GL_6(F_\nu)\times \GL_{1}(F_\nu)$.
Its Satake parameter of the unramified constituent is represented by
$$
(\diag\cpair{\chi_{1}(\varpi_\nu),\chi_{2}(\varpi_\nu),\dots,\chi_{6}(\varpi_\nu)},\chi_{0}(\varpi_\nu),1).
$$

The unramified local twisted exterior cube $L$-function of $\pi_\nu$ is defined as the following: 
when $\nu$ is inert, $L(s,\pi_\nu,\e{3})^{-1}$ equals  
\begin{align}
&(1-\chi_1\chi_2\chi_3\chi_0q^{-s})(1-\chi_1\chi_2\chi_0q^{-s})(1-\chi_1\chi_3\chi_0q^{-s})(1-\chi_2\chi_3\chi_0q^{-s})\nonumber \\
\times&(1-\chi_1\chi_0q^{-s})(1-\chi_2\chi_0q^{-s})(1-\chi_3\chi_0q^{-s})(1-\chi_0q^{-s})\nonumber\\
\times&(1-\chi^2_1\chi_2\chi_3\chi^2_0q^{-2s})(1-\chi_1\chi^2_2\chi_3\chi^2_0q^{-2s})(1-\chi_1\chi_2\chi^2_3\chi^2_0q^{-2s})\nonumber\\
\times&(1-\chi_1\chi_2\chi^2_0q^{-2s})(1-\chi_2\chi_3\chi^2_0q^{-2s})(1-\chi_1\chi_3\chi^2_0q^{-2s});    \label{eq:L-factor}
\end{align}
and when $\nu$ is split,
$$
L(s,\pi_\nu,\e{3})=\prod_{i<j<k}(1-\chi_i\chi_j\chi_k\chi_0q^{-s}),
$$
where $\chi_i(\varpi_\nu)$ is written as $\chi_i$ for short
and $q$ is the cardinality of the residue field of $F_\nu$.

\subsection{Unramified calculation}\label{sec:unramified-zeta}


In this section, we will compute the local zeta integral $\CZ_\nu(s,W_\varphi,\phi_\nu)$ over all unramified places. 
Suppose that all the data are unramified, i.e $\nu\notin S$. 
We will drop the subscript $\nu$ in the rest of this  section if no confusion is caused. 

Assume that $\pi_\nu$ is the unramified constituent of $I(\chi_{1}\otimes\chi_{2}\otimes\chi_{3}\otimes\chi_{0})$ for some unramified characters $\chi_i$ as in Section \ref{sec:local-L-factor}.
Let $W_\varphi$ be the unramified Whittaker function $W_\pi$ of  $\pi_\nu$ 
and $\phi_\chi$ be the spherical function in the normalized induced representation 
$$\Ind^{H(F_\nu)}_{P(F_\nu)}|a|^{\frac{3}{2}s}|\det g|^s\omega_\pi\chi^3(a)\omega_\pi\chi^2(\det g)$$
where the values of $W_\pi$ and $\phi_\chi$ at identity are normalized as 1.
Following the normalization \eqref{eq:normalized-zeta-section-2},
one has
\begin{equation}\label{eq:normalized-zeta}
\CZ^*_\nu(s,W_\pi,\phi_\chi)=
\zeta(s+2,\omega_\pi\chi^2)\zeta(2s+2,\omega_\pi\chi^2)
\CZ_\nu(s,W_\pi,\phi_\chi),	
\end{equation}
where $\zeta(\cdot)$ is the local $\zeta$-factor.

\begin{pro} \label{pro:unramified}
Over a unramified  place $\nu$, when $\Re(s)$ is sufficiently large,  one has
\begin{equation}\label{eq:unramified}
\CZ^*_\nu(s,W_\pi,\phi_\nu)=L(\frac{s}{2}+\frac{1}{2},\pi_\nu\otimes\chi_\nu,\e{3}).	
\end{equation}
\end{pro}

If $\nu$ is split, the calculation is the same with Proposition 2.1 in \cite{GR00}.
Hence we will focus on the non-split case here. 
Suppose that $\nu$ is inert in the rest of Section \ref{sec:unramified-zeta}.

Note that the modular character of the standard Borel subgroup $B_H$ of $\GSp_6$ is given by
$\delta_H(g)=|a^4b^6c^{10}|$ for  
\begin{equation}\label{eq:t-abc}
g=\diag\{abc,ac,a,1,c^{-1},b^{-1}c^{-1}\}.	
\end{equation}
Applying the Iwasawa decomposition and the argument similar to Section 2.1 \cite{GR00}, 
we can reduce to 
\begin{align}
\CZ(\cdot)=&\int_{(F^\times)^3}\int_{F}
W_\pi(\iota(v_3(x))\cdot g)\omega_\pi(bc^2)\chi(ab^2c^4)|a^2b^4c^8|^{\frac{s+2}{4}}|a^5b^6c^{10}|^{-1}\ud x\ud g \nonumber \\
=&\int_{(F^\times)^3}W_{\pi}(g)\omega_\pi(bc^2)\chi(ab^2c^4)|a^2b^4c^8|^s|a^5b^6c^{10}|^{-1}K(a)\ud a\ud b\ud c.\label{eq:zeta-unram}
\end{align}
Here the {\it generating function} $K(a)$ is the evaluation of the integral over $x\in F$, which is 
$$
K(a)=\frac{\zeta(s+1,\omega_\pi\chi^2)}{\zeta(s+2,\omega_\pi\chi^2)}
(1-\omega_\pi\chi^2(a)|a|^{s+1}\omega_\pi\chi^2(\varpi)q^{-s-1}).
$$

\subsubsection{Casselman-Shalika formula for $G$}\label{sec:C-S-formula}
To carry on computing $\CZ(\cdot)$ in \eqref{eq:zeta-unram}, we need the  Casselman-Shalika formula for $W_\pi(g)$ of $G(F)$ on the diagonal torus of $G$.
Tamir \cite{T91} explicitly wrote down the Casselman-Shalika formula for the unitary group case.
In \cite{GH06}, Gan and Hundley  gave an interpretation of the Casselman-Shalika formula for the quasi-split groups in terms of certain dual groups.
Following those results, Furusawa and Morimoto \cite{FM13} wrote the Casselman-Shalika formula for $\GU_4$ in terms of the dual group of $\GSp_4$.
Similar to the $\GU_4$ case, we adopt the notation in Page 4135 of \cite{FM13} to give the Casselman-Shalika formula for $\GU_6$ in terms of the dual group  $\GSpin_7(\BC)$ of $\GSp_6$.
All unexplained notation in this section can be found in \cite{FM13}.

Let $T_G$ and $T_H$ be the diagonal tori of $G$ and $H$ respectively,
Denote $\varsigma^*$ to be the homomorphism from ${}^LT^\circ_G$ to ${}^LT_H$ induced from the inclusion map $\varsigma\colon T_H\mapsto T_G$,
where ${}^LT^\circ_G\subset \GL_6(\BC)\times\GL_1(\BC)$ is the connected component of the $L$-group ${}^LT^\circ_G$.
Write 
$$s_\pi=(\diag(\chi_1(\varpi),\chi_2(\varpi),\chi_3(\varpi),1,1,1),\chi_0(\varpi),1)\in {}^LT^\circ_G.$$
By \eqref{eq:Satake}, we have $t_\pi=(s_\pi,\sig)\in {}^LG$ and  $\bar{s}_\pi:=\varsigma^*(s_\pi)$ is in ${}^LT_H$.
Similar to (3.9) in \cite{FM13}, 
if $g$ of form \eqref{eq:t-abc} modulo the compact subgroup $T_H(\Fo)$ is of form
$$
d(n_1,n_2,n_3):=\diag(\varpi^{n_1+n_2+n_3},\varpi^{n_1+n_3},\varpi^{n_1},1,\varpi^{-n_3},\varpi^{-n_2-n_3}) 
$$ 
for some non-negative integers $n_i$,
then $W_\pi(g)$ equals the value of the character of the irreducible representation of $\GSpin_7(\BC)$ of the highest weight given by $n_i$ ($1\leq i\leq 3$) at $\bar{s}_\pi$.
Otherwise, $W_\pi(g)=0$.

For convenience, we may express $W_\pi(g)$ as the value of the character of the irreducible representation of $\Spin_7(\BC)$.
In fact, we may choose a unramified character $\chi'_0$ such that $\chi_1\chi_2\chi_3\chi'^2_0=1$ is the trivial character, and 
set
\begin{equation} \label{eq:s'}
\bar{s}'_\pi=\varsigma^*(\diag(\chi_1(\varpi),\chi_2(\varpi),\chi_3(\varpi),1,1,1),\chi'_0(\varpi),1) 	
\end{equation}
Then $\bar{s}'_\pi$ is of similitude 1 and belongs to $\Spin_7(\BC)$.
As $\omega_\pi$ is given by \eqref{eq:center-character},  we have
\begin{equation}\label{eq:W}
W_\pi(d(n_1,n_2,n_3))=\delta^{\frac{1}{2}}_G(g)\ch(n_2,n_3,n_1)(\bar{s}'_\pi)\omega_\pi(\varpi)^{\frac{n_1}{2}}. 	
\end{equation}
Here 
$\delta_G$ is the modular character of the Borel subgroup $B_G$ of $G$,
 $\ch(n_2,n_3,n_1)$ is  the character of the representation of ${\rm Spin}_7(\BC)$ of highest weight $(n_2,n_3,n_1)$,
and $\ch(n_2,n_3,n_1)(\bar{s}'_\pi)$ is the value of $\ch(n_2,n_3,n_1)$ at $\bar{s}'_\pi$.

Note that
\begin{equation}\label{eq:delta-G}
\delta_G(\diag\{ax,ay,az,\bar{z}^{-1},\bar{y}^{-1},\bar{x}^{-1}\})=
|a|^9|x|_{E}^{5}|y|_{E}^{3}|z|_{E},	
\end{equation}
where $|x|_E=|N_{E/F}(x)|$.
If $g$ is of form \eqref{eq:t-abc} in $T_H(F)$, then $\delta_G(g)=|a^9b^{10}c^{16}|$.

Plugging \eqref{eq:W} and \eqref{eq:delta-G} into \eqref{eq:zeta-unram}, we have 
\begin{align*}
\CZ_\nu(\cdot)= & \frac{\zeta(s+1,\omega_\pi\chi^2)}{\zeta(s+2,\omega_\pi\chi^2)}\sum_{n_1,n_2,n_3=0}^{\infty}\ch(n_2,n_3,n_1)\chi(\varpi)^{n_1+2n_2+4n_3}\omega_\pi(\varpi)^{\frac{n_1}{2}+n_2+2n_3}\\
&\times 
q^{(-s/2-1/2)n_1+(-s-1)n_2+(-2s-2)n_3}(1-(\omega_\pi\chi^2)(\varpi)^{n_1+1}q^{(-s-1)(n_1+1)}).
\end{align*}
Here we use $\ch(n_2,n_3,n_1)$ as shorthand for $\ch(n_2,n_3,n_1)(t_\pi)$.

For simplicity, write $t=\omega_\pi(\varpi)^{\frac{1}{2}}\chi(\varpi)q^{-\frac{s+1}{2}}$.
By definition \eqref{eq:normalized-zeta}, Equation \eqref{eq:unramified} is equivalent to
\begin{equation}\label{eq:CZ-1}
\sum_{n_1,n_2,n_3=0}^{\infty}\ch(n_2,n_3,n_1)  t^{n_{1}+2n_{2}+4n_{3}}
\frac{1-t^{2(n_{1}+1)}}{1-t^{2}}  
=(1-t^4)L(\frac{s}{2}+\frac{1}{2},\pi\otimes\chi,\e{3}).	
\end{equation}

On the other hand,
let $\chi'$ be the  restriction of $\chi_1\otimes\chi_2\otimes\chi_3\otimes\chi'_0$ into $T_H(F)$ where $\chi'_0$ satisfies $\chi_1\chi_2\chi_3\chi'^2_0=1$ and $\pi_H$ be the unramified constituent of $\Ind^{H(F)}_{B_H(F)}\chi'$.
Since $\pi_H$ is of the trivial central character, 
we may also consider $\pi_H$ as a representation of $\PGSp_6$, whose $L$-group is $\Spin_7(\BC)$. 
And the Satake parameter of $\pi_H$ is $\bar{s}'_\pi$ defined in \eqref{eq:s'}.
Hence, we may write $L(\frac{s+1}{2},\pi\otimes\chi,\e{3})$ as follows
\begin{equation}\label{eq:L-St-Spin}
L(\frac{s+1}{2},\pi_H\otimes\chi,\Spin)
\cdot L(s+1,\pi_H\otimes\chi^2,\St)\cdot (1-\omega_\pi\chi^2 q^{-(s+1)}),	
\end{equation}
where  $\St$ and $\Spin$ are the standard and Spin representations of $\Spin_7(\BC)$ associated to the fundamental weights $\omega_1$ and $\omega_3$ respectively. 
For the reader's convenience, we recall  
\begin{align*}
 &L(\frac{s+1}{2},\pi_H\otimes\chi,\Spin)^{-1}\\
 =&
 (1-(\chi_1\chi_2\chi_3)^{\frac{1}{2}}\cdot t)(1-(\chi^{-1}_1\chi_2\chi_3)^{\frac{1}{2}}\cdot t)(1-(\chi_1\chi^{-1}_2\chi_3)^{\frac{1}{2}}\cdot t)\\
 &\times (1-(\chi_1\chi_2\chi^{-1}_3)^{\frac{1}{2}}\cdot t)
 (1-(\chi_1\chi_2\chi_3)^{-\frac{1}{2}}\cdot t)(1-(\chi^{-1}_1\chi_2\chi_3)^{-\frac{1}{2}}\cdot t)\\
 &\times (1-(\chi_1\chi^{-1}_2\chi_3)^{-\frac{1}{2}}\cdot t)(1-(\chi_1\chi_2\chi^{-1}_3)^{-\frac{1}{2}}\cdot t)
\end{align*}
and 
$$
L(s+1,\pi_H\otimes\chi^2,\St)^{-1}=(1-t^2)\prod^{3}_{i=1}(1-\chi_it^2)(1-\chi^{-1}_it^2).
$$
Remark that our above calculation is independent of the choices of $\chi'_0$ as only $\chi'^{2}_0$ is involved.

Following from \eqref{eq:CZ-1} and \eqref{eq:L-St-Spin}, Equation \eqref{eq:unramified} is equivalent to
\begin{align}
&\sum_{n_1,n_2,n_3=0}^{\infty}\ch(n_2,n_3,n_1)  t^{n_{1}+2n_{2}+4n_{3}}
\frac{1-t^{2(n_{1}+1)}}{1-t^{2}}\nonumber \\ 
=&(1-t^2)(1-t^4)L(\frac{s+1}{2},\pi_H\otimes\chi,\Spin)
L(s+1,\pi_H\otimes\chi^2,\St).	\label{eq:CZ-2}
\end{align}

Now, applying the Poincar\`e identity for the right hand side of \eqref{eq:CZ-2}, we have
\begin{align*}
 & L(s+1,\pi_H\otimes\chi^2,\St)=\sum^\infty_{k=0}\ch_{\Sym^k(\omega_1)}(\bar{s}'_\pi)t^{2k}\\
& L(\frac{s+1}{2},\pi_H\otimes\chi,\Spin)=
\sum^\infty_{m=0}\ch_{\Sym^m(\omega_3)}(\bar{s}'_\pi)t^m,
\end{align*}
where $\Sym^n$ is the symmetric $n$-th power operation and $\ch_{\Sym^n}$ is the character of the corresponding representation. 
Following Theorem 3.2 and the table in Page 13 of \cite{Br83}, one has the decomposition
\begin{align*}
\ch_{\Sym^k(\omega_1)}(\bar{s}'_\pi)
=&\sum_{\substack{2i+k_1=k\\ i,~k_1\geq 0}}\ch(k_1,0,0)(\bar{s}'_\pi)=:A_k\\
\ch_{\Sym^m(\omega_3)}(\bar{s}'_\pi)
=&\sum_{\substack{2j+m_1=m\\ j,~m_1\geq 0}}\ch(0,0,m_1)(\bar{s}'_\pi)=:B_m.
 \end{align*}
Set  $C_r:=\sum_{\substack{2k+m=r\\ k,~m\geq 0}}A_k\cdot B_m$ and then 
\begin{align}
 & (1-t^2)(1-t^4)L(s+1,\pi_H\otimes\chi^2,\St)L(\frac{s+1}{2},\pi_H\otimes\chi,\Spin)\nonumber \\
=&\sum^\infty_{r=0}(C_{r}-C_{r-2}-C_{r-4}+C_{r-6})t^{r}. \label{eq:L-C}
\end{align}
Default $C_{r}=0$ for $r<0$.

By simple manipulations, we have
\begin{equation}\label{eq:C}
C_{r}-C_{r-2}-C_{r-4}+C_{r-6}=\sum_{\substack{2k+m=r\\ k,~m\geq 0}}\ch(k,0,0)(\bar{s}'_\pi)\cdot \ch(0,0,m)(\bar{s}'_\pi).	
\end{equation}
Hence, combining \eqref{eq:L-C} and \eqref{eq:C} and continuing \eqref{eq:CZ-2}, 
to prove Proposition \ref{pro:unramified}, it is enough to show 
\begin{align}
&\sum_{n_1,n_2,n_3=0}^{\infty}\ch(n_2,n_3,n_1)  t^{n_{1}+2n_{2}+4n_{3}}
\frac{1-t^{2(n_{1}+1)}}{1-t^{2}}\nonumber \\ 
=&(\sum^\infty_{m=0}\ch(0,0,m)t^{m})(\sum^\infty_{k=0}\ch(k,0,0)t^{2k}).	\label{eq:L-characters}
\end{align}

\subsubsection{ Pieri's formula for $\Spin_7(\BC)$}
In this section, we apply an analogue of Pieri's formula for the $\Spin$ group 
to compute the character $\ch(0,0,m)\cdot\ch(k,0,0)$ and complete the proof of Proposition \ref{pro:unramified}. 
Denote $\rho(k_1,k_2,k_3)$ to be the irreducible representation of $\Spin_7(\BC)$ with the highest weight $\sum^3_{i=1}k_i\omega_i$. 
The analogue of Pieri's formula  is applied to decompose
$$
\rho(\alpha)\otimes\rho(k,0,0)=\sum_{\beta} m_{\alpha\vert\beta} \rho(\beta),
$$
where $\alpha$ and $\beta$ are dominant weights and $m_{\alpha\vert\beta}$ is the multiplicity.
In \cite{BFG99}  Bump, Friedberg and Ginzburg studied this type of formula for $\Spin_9(\BC)$ to compute the local factors $L(s_1,\pi,\omega_1)L(s_2,\pi,\omega_3)$ of $\GSp_8$. 
We adopt their approach in Page 756--760 \cite{BFG99} to our case and obtain the following lemma.

\begin{lm}\label{lm:decomposition}
We have the following decompositions:
\begin{align}
 &\rho(0,0,2m)\otimes\rho(k,0,0)\nonumber\\
=&\sum^1_{\eps=0}\sum_{r=0}^{\min\{k-\eps,m-\eps\}}
 \sum_{a=0}^{\min\{k-\eps-r,r\}} \rho(k-\eps-a-r,r-a,2(m-r+a)) \label{eq:lm-Pieri-even}
\end{align}
and
\begin{align}
 &\rho(0,0,2m+1)\otimes\rho(k,0,0)\nonumber\\
=&\sum^1_{\eps=0}\sum_{r=0}^{\min\{k-\eps,m\}}
 \sum_{a=0}^{\min\{k-\eps-r,r\}} \rho(k-\eps-a-r,r-a,2(m-r+a)+1). \label{eq:lm-Pieri-odd}
\end{align}
\end{lm}

\begin{proof}
The proof of this lemma is similar to Proposition 3.5 \cite{BFG99}.
Our goal is to obtain the decomposition $\rho(\alpha)\otimes\rho(k,0,0)$  for $\alpha=(0,0,n)$, i.e., 
\begin{equation}\label{eq:lm-decomposition}
\rho(0,0,n)\otimes\rho(k,0,0)=\sum_\beta m_{\alpha\vert\beta}\rho(\beta).	
\end{equation}
Note that $\rho(k_1,k_2,k_3)$ factors through $\SO_{7}(\BC)$ when $k_3$ is even and it is a spinor representation when $k_3$ is odd.
If $\beta=(k_1,k_2,k_3)$ occurs in \eqref{eq:lm-decomposition} then $k_3$ has the same parity with $n$.

First, let us verify \eqref{eq:lm-Pieri-even}.
For the convenience of the reader, we recall Sundaram's algorithm in \cite{Su} for $m_{\alpha\vert\beta}$ when $n$ is even in \eqref{eq:lm-decomposition}.
For each irreducible representation $\rho(\beta)$ with $\beta=(k_1,k_2,k_3)$ with $m_{\alpha\vert\beta}\ne 0$, we have $k_3$ is even in this case and  
 associate $\beta$ with the partition  
$$
\tilde{\beta}:=[k_1+k_2+ k_3/2, k_2+k_3/2,  k_3/2],
$$ 
which  also gives a Young diagram.
So  the associated partition of $\alpha=(0,0,2m)$ is $\tilde{\alpha}=[m,m,m]$.
Then $m_{\alpha\vert\beta}$ equals the number of partitions  $\tilde{\nu}$
such that
\begin{itemize}
	\item $\tilde{\nu}$ are contained in $\tilde{\alpha}$ and $\tilde{\beta}$ as a Young diagram;
	\item $\tilde{\alpha}\bks \tilde{\nu}$ and $\tilde{\beta}\bks \tilde{\nu}$ are horizontal strips;
	\item either (i) $|\tilde{\alpha}\bks \tilde{\nu}|+|\tilde{\beta}\bks \tilde{\nu}|=k$
	or (ii) $\ell(\tilde{\nu})=3$ and $|\tilde{\alpha}\bks \tilde{\nu}|+|\tilde{\beta}\bks \tilde{\nu}|=k-1$. 
\end{itemize}
Here $\ell(\tilde{\nu})$ is the length of the partition $\tilde{\nu}$ and $|\tilde{\alpha}\bks\tilde{\nu}|$ is the weight of the partition, that is, the total number of boxes in the horizontal strip $\tilde{\alpha}\bks\tilde{\nu}$.

In our case,  $\tilde{\beta}$ and $\tilde{\nu}$ are partitions of the length $\leq 3$. 
Assume that $\tilde{\nu}$ satisfies the above conditions for $\tilde{\alpha}$ and $\tilde{\beta}$.
Since $\tilde{\nu}$ is contained in $\tilde{\alpha}$ and $\tilde{\alpha}\bks \tilde{\nu}$ is horizontal, we may assume 
$$
\tilde{\nu}=[m,m,m-r]
$$ 
with $0\leq r\leq m$.
By the other two conditions on $\tilde{\nu}$ and $\tilde{\beta}$,
$\tilde{\beta}$ must be of form 
$$
\tilde{\beta}=[m+k-\eps-r-a,m,m-r+a]
$$
where $\eps\in\{0,1\}$ $0\leq a\leq \{r,k-\eps-r\}$.
Note that $\eps=1$ only happens when $r\leq m-1$, which implies $r\leq m-\eps$. 
Then we obtain the constraints of $a,r,\eps$ on the right hand side of \eqref{eq:lm-Pieri-even}.

Moreover, given $k$ and $m$, the partition $\tilde{\beta}$ is distinct when varying $a$ and $r$. It follows that the decomposition is multiplicity-free, i.e., $m_{\alpha\vert\beta}\leq 1$. 
Since the representation  corresponding to $\tilde{\beta}$ is $\rho(k-\eps-r-a,r-a,2(r-a))$,
 we conclude \eqref{eq:lm-Pieri-even}.

Finally, to prove \eqref{eq:lm-Pieri-odd}, we apply Proposition 3.4 of \cite{BFG99} for the spinor representations and  obtain the description of $\tilde{\nu}$ and $\tilde{\beta}$.
Then \eqref{eq:lm-Pieri-odd} follows from a similar argument of \eqref{eq:lm-Pieri-even}. We omit the details here. 
\end{proof}

\begin{proof}[Proof of Proposition \ref{pro:unramified}]
In Section \ref{sec:C-S-formula}, we reduce Proposition \ref{pro:unramified} to verify Equation \eqref{eq:L-characters}, which is equivalent to
\begin{align}
&\sum_{n_1,n_2,n_3=0}(\sum^{n_1}_{i=0}t^{n_{1}+2i+2n_{2}+4n_{3}})\ch(n_2,n_3,n_1) \nonumber\\
=&\sum_{k,m\geq 0}\sum_{a,r,\eps}\ch(k-\eps-a-r,r-a,m-2r+2a)t^{m+2k}\label{eq:CZ-3}
\end{align}
where $a,r,\eps$ are given in \eqref{eq:lm-Pieri-even} and \eqref{eq:lm-Pieri-odd} according to the parity of $m$.
  
By comparing the coefficients in $\BC[t]$ of characters $\ch(n_2,n_3,n_1)$ on the both sides,
it suffices to show that the coefficient of $\ch(n_2,n_3,n_1)$ in the right hand side \eqref{eq:CZ-3}
is equal to $\sum^{n_1}_{i=0}t^{n_{1}+2i+2n_{2}+4n_{3}}$.

Suppose that $n_1$ is even.
Let us collect all the coefficients of the character of $\rho(k-\eps-a-r,r-a,m'-r+a)$ satisfying 
$$
(k-\eps-a-r,r-a,m-2r+2a)=(n_2,n_3,n_1),
$$
which is equivalent to
\begin{equation}\label{eq:k-m-r}
k=n_2+n_3+\eps+2a,\quad m=n_1+2n_3,\quad r=n_3+a.	
\end{equation}
If $n_1$ is even, by \eqref{eq:k-m-r} the constraints on $a,r$ are equivalent to $0\leq a\leq n_1/2-\eps$;
If $n_1$ is odd, they are equivalent to $0\leq a\leq \frac{n_1-1}{2}$.
Then the coefficient of $\ch(n_2,n_3,n_1)$ in the right hand side \eqref{eq:CZ-3} is given by
$$
\begin{cases}
\sum^1_{\eps=0}\sum_{a=0}^{\frac{n_1}{2}-\eps}t^{n_1+2n_2+4n_3+2(2a+\eps)}& \text{ if $n$ is even}\\
\sum^1_{\eps=0}\sum_{a=0}^{\frac{n_1-1}{2}}t^{n_1+2n_2+4n_3+2(2a+\eps)}& \text{ if $n$ is odd},
\end{cases}
$$
which equals $\sum^{n_1}_{i=0}t^{n_{1}+2i+2n_{2}+4n_{3}}$ in all cases.

Hence, we complete the proof of Proposition \ref{pro:unramified}.
\end{proof}

\section{Poles of exterior cube $L$-function}\label{sec:analytic-L}

In this section, we assume that $F$ is a local field and show that the local zeta integrals $\CZ_\nu(s,W_\varphi,\phi)$ over all local places has the desired analytic properties such as
the convergence and their meromorphic continuation.
When $G$ is split over $F$, the case has been considered in \cite{GR00}.
It suffices to prove the analogous results for the non-split case.
Throughout this section, assume that $E$ is not split and then $G(F)\cong \GU_6(F)$ is quasi-split.

\begin{lm}\label{lm:local-analytic}
Let $\phi$ be a standard section and $W_\varphi\in \CW(\pi,\psi)$ be a smooth Whittaker function.
Then
\begin{enumerate}
	\item \label{item:LA-convergent} $\CZ(s,W_\varphi,\phi)$ converges absolutely for $\Re(s)$ sufficiently large;
	\item \label{item:LA-meromorphic} It continues to a meromorphic function in the whole plane;
	\item \label{item:LA-nonzero} For any $s_0\in \BC$, there exists a choice of a smooth function $W\in \CW(\pi,\psi)$ and a $K_H$-finite function $\phi$ such that $\CZ(s,W,\phi)$ is not zero at $s=s_0$.  
	\item \label{item:LA-nonarchimedean}
	If $F$ is a nonarchimedean local field, $\CZ(s,W_\varphi,\phi)$ is a rational function in $q^{-s}$ and can be made non-zero constant for some $W\in \CW(\pi,\psi)$.
\end{enumerate}
\end{lm}

The idea to prove this lemma is the same as in the proof of Proposition 2.2 of \cite{GR00}.
For our case $\GU_6$, the main differences are the manipulations on the unipotent subgroups for $\GU_6$ on the way to prove Item \eqref{item:LA-nonzero},
and the asymptotic expansion  of smooth Whittaker functions $W_\varphi$ used in the proof,
which is given in the following Lemma.  
So we omit the proof of Lemma \ref{lm:local-analytic} here. 

\begin{lm}\label{lm:Whittaker-expansion}
There exists a finite set $X$ of finite functions on $(E^\times)^2\times F^\times$ such that for any smooth function $W$ in $\CW(\pi,\psi)$, for every $\xi\in X$ there exist Schwartz-Bruhat functions $\theta_\xi\in \CS( E^2 \times F)$ such that
\begin{equation}\label{eq:Whittaker-expansion}
W(t)=\sum_{\xi\in X}\theta_\xi(a,b,c)\xi(a,b,c).
\end{equation}
where $t=\diag\{abc,ac,a,1,\bar{c}^{-1},(\bar{b}\bar{c})^{-1}\}\in T_G(F)$.
\end{lm} 
The proof of Lemma \ref{lm:Whittaker-expansion} is similar to many cases  in literatures.
For example, if $F$ is a non-archimedean local field, see Proposition 2.2 \cite{So93} for non-archimedean fields and Section 4 in \cite{So95} for archimedean cases.
Also such an expansion has been considered in \cite{LM09} for connected split reductive groups over non-archimedean fields.
Here we omit the details of the proof.

\begin{thm}
Let $\pi$ be an irreducible generic cuspidal automorphic representation of $\GU_6(\BA)$. Then
$L^S(s,\pi\otimes\chi,\e{3})$ is entire unless $\omega_\pi\chi^2$ is a nontrivial quadratic character of $\GL_1(F)\bks \GL_1(\BA)$,
in which case $L^S(s,\pi\otimes\chi,\e{3})$ can have at most a simple pole at $s=0$ or $s=1$.
\end{thm}
\begin{proof}
Suppose that $\phi_{\chi_\pi}$ is  a standard $K_H$-finite section and unramified outside of $S$.
First, let us study the possible poles of $\CZ^*(s,\varphi_\pi,\phi_{\chi_\pi})$.

Following Lemma \ref{lm:pole-Eisenstein}, the normalized Eisenstein series $E^*(\phi_{\chi_\pi},\lam_s)$ has at most a simple pole at $s=1$ and $s=2$ for $\Re(s)\geq 0$, so is the normalized zeta integral $\CZ^*(s,\varphi_\pi,\phi_{\chi_\pi})$.

When  $E^*(\phi_{\chi_\pi},\lam_s)$ has a pole at $s=2$, its residue is a character of $H(\BA)$ (see Lemma 5.5 \cite{G95} for instance).
Similar to the proof of Theorem 1 in \cite{DP16}, the following integral is identically zero,
$$
\int_{Z_H(\BA)H(F)\bks H(\BA)}\varphi_\pi(g)\omega_\pi(g)^{-1}\ud g\equiv 0.
$$
Hence for all $\omega_\pi$ and $\chi$, we have $\CZ^*(s,\varphi_\pi,\phi_{\chi_\pi})$ is entire at $s=2$.

Suppose that $\omega_\pi\chi^2=1$.
For this case, the argument is similar to Lemma 5.1 of \cite{G95},  
Denote $Q=M_QV_Q$ to be the Klingen mirabolic subgroup of $H$, whose Levi subgroup $M_Q$ consists of elements of form
$$
m_Q(a,g):=\begin{pmatrix}
a&&\\&g&\\&&\lam(g)a^{-1} 	
\end{pmatrix}
$$
where $a\in\GL_1$, $g\in \GSp_4$ and $\lam(g)$ is the similitude of $g$.
Define $\chi_Q(a,g)=\chi\circ\lam(g)$ to be the character of $M_Q(\BA)$.
Similar to Section \ref{sec:global}, one can form an forms the Eisenstein series on $H(\BA)$ for $\lam_{Q,s}\in X_{M_Q}$ 
$$
E_Q(\phi_{\chi_Q},\lam_{Q,s})(g)=\sum_{\gamma\in Q(F)\bks H(F)}\lam_{Q,s}\phi_{\chi_Q}(\gamma g)
$$
where $\phi_{\chi_Q}\in \CA(V_Q(\BA)M_Q(F)\bks H(\BA))_{\chi_Q}$.

By Lemma 5.5 \cite{G95} for instance, when $\omega_\pi\chi^2=1$, one has the residue of $E^*(\phi_{\chi_\pi},\lam_s)$ at $s=1$ is proportional to $E_Q(\phi'_{\chi_Q},\lam_{Q,s_0})$ for some $\phi'_{\chi_Q}$ and $s_0\in \BC$.
It suffices to show that
\begin{equation}\label{eq:vanishing}
\int_{Z_H(\BA)H(F)\bks H(\BA)}\varphi_\pi(g)E_Q(\phi'_{\chi_Q},\lam_{Q,s})(g)\ud g\equiv 0
\end{equation}
for all $\phi_{\chi_Q}$ and $s$.

For $\Re(s)$ large, we may unfold the Eisenstein series $E_Q$ to obtain
$$
\int_{Q(\BA)\bks H(\BA)}\int_{Z_H(\BA)Q(F)\bks Q(\BA)}\varphi_\pi(g)\lam_{Q,s}\phi'_{\chi_Q}(g)\ud g.
$$
Following  a similar argument of \cite[Lemma 5.1]{G95} or \cite[Theorem 1]{DP16}, we have the inner integral $\int_{Z_H(\BA)Q(F)\bks Q(\BA)}$ is identically zero.
By meromorphic continuation on $s$, we obtain \eqref{eq:vanishing}. 
Thus, if $\omega_\pi\chi^2=1$,  $\CZ^*(s,\varphi_\pi,\phi_{\chi_\pi})$ is entire at $s=1$.

Overall,  
if $\omega_\pi\chi^2=1$ or $\omega^2_\pi\chi^4\ne 1$, then $\CZ^*(s,\varphi_\pi,\phi_{\chi_\pi})$ is entire;
if $\omega^2_\pi\chi^4=1$ and $\omega_\pi\chi^2\ne 1$, then $\CZ^*(s,\varphi_\pi,\phi_{\chi_\pi})$ has at most a simple pole at $s=-1$ or $s=1$. 
By Propositions \ref{pro:Eulerian} and \ref{pro:unramified}, we have
\begin{equation}\label{eq:CZ-L}
\CZ^*(s,\varphi_\pi,\phi_{\chi_\pi})=\prod_{\nu\in S}\CZ^*_\nu(s,W_\pi,\phi_\nu)
\cdot L^S(\frac{s}{2}+\frac{1}{2},\pi\otimes\chi,\e{3}).	
\end{equation}
According to Lemma \ref{lm:local-analytic}, for each $\nu\in S$ there exist $\phi_{\nu}$ and $W_\nu\in \CW(\pi_\nu,\psi_\nu)$ such that $\CZ^*(s,\varphi_\pi,\phi_{\chi_\pi})$ is nonzero. 
This theorem follows from the locations of the poles of $\CZ^*(s,\varphi_\pi,\phi_{\chi_\pi})$ mentioned above and \eqref{eq:CZ-L}.
\end{proof}

Next, analogous to Theorem 3.3 of \cite{GR00}, we may use the automorphic period integral to capture the residue of $L^S(s,\pi\otimes\chi,\e{3})$ at $s=1$ for the $\GU_6$ case,
when $\omega_\pi\chi^2$ is a non-trivial quadratic character.
We apply Siegel-Weil formula in \cite{KR94} to replace the residue of $E(\phi_{\chi_\pi},\lam_s)$ at $s=1$ in terms of the theta correspondence, which part is exactly the same as the $\GL_6$ case in \cite{GR00}.
Then one may obtain easily an analogue of Theorem 3.3 of \cite{GR00}.

Let $\theta_f$ be the theta function of the metaplectic group $\wt{\Sp}_{12}(\BA)$,
where $f\in\CS(\BA^6)$ the Schwartz space on $\BA^6$.
Recall that $E^+$ is the quadratic extension of $F$ associated to $\omega_\pi\chi^2$,
and denote $\RO_2$ to be the orthogonal group defined by the norm form on $E^+$,
whose $F$-points are the $\BZ_2$-extension of the norm one one elements of $E^+$.
Referring to \cite[Corollary 6.3]{KR94},
$$
\Res_{s=1}E(\phi_{\chi_\pi},\lam_s)(g)=\int_{\RO_2(F)\bks \RO_2(\BA)}\theta_f(g,h)\ud h.
$$ 
Following  \eqref{eq:CZ-L}, we use the automorphic period integrals to characterize the poles of the exterior cube $L$-functions.
\begin{cor}\label{cor:pole}
Assume that $\omega_\pi\chi^2$ is a non-trivial quadratic character.
If $L^S(s,\pi\otimes\chi,\e{3})$ has a pole at $s=1$, then there exits a choice of $\varphi_\pi\in\pi$ and $f\in \CS(\BA^6)$ such that
$$
\int_{\Sp_6(F)\bks\Sp_6(\BA)}\int_{\RO_2(F)\bks \RO_2(\BA)}\varphi_\pi(g)\theta_f(g,h)\ud g\ud h
$$
is nonzero.
\end{cor}

\section{Global tempered $L$-packets of $\GU_n$}\label{sec:endoscopic-classificiation}

Different from Corollary \ref{cor:pole}, in next two sections we will use the endoscopic classification and the automorphic induction to characterize the existence of poles of the exterior cube $L$-functions. 
To do that, first we extend the endoscopic classification of $\RU_n$ in \cite{Mk15} to the quasi-split unitary similitude groups $\GU_n$, 
via the restriction of representations from $\GU_n$ to $\RU_n$, following \cite{Xu16} and \cite{Xu17}.

In general, let $\CG$ be a connected reductive group defined over $F$ and $Z_{\CG}$ be the center of $\CG$.
For a character $\chi$  of $Z_{\CG}(F)\bks Z_{\CG}(\BA)$,
denote by $L^2(\CG(F)\bks\CG(\BA),\chi)$  the space of $\chi$-equivariant $L^2$-functions over $\CG(F)\bks\CG(\BA)$, 
and $L^2_\disc(\CG,\chi)$ to be its discrete spectrum.
Let $\CA_\disc(\CG)$ the set of equivalence classes of irreducible unitary representations $\pi$ of $\CG(\BA)$ occurring in the discrete spectrum $L^2_\disc(\CG,\chi)$. 
The theory of endoscopic classification for classical groups  is to parameterize the discrete automorphic representations by means of the global Arthur parameters,
which can be realized as certain automorphic representations of general linear groups (their Langlands functorial lifts).

In \cite{A13}, Arthur established the local Langlands correspondence for irreducible admissible representations over local fields and the endoscopic classification of discrete automorphic representations when $\CG$ are symplectic groups and quasi-split orthogonal groups.
Mok (\cite{{Mk15}}) extended the results of Arthur to the quasi-split unitary group case. 

Let $\CG\subseteq \wt{\CG}$ be two connected reductive groups and of the same derived group  over $F$.  
By using the endoscopic theory, Xu gave a conjectural description of local $L$-packets of $\wt{\CG}$ based on the local $L$-packets of $\CG$ in \cite{Xu16} and a conjectural endoscopic classification of tempered automorphic representations of $\wt{\CG}$ in \cite{Xu17}.
Following the work of Arthur, Xu verified those conjectures for symplectic and orthogonal similitude groups.

In this section, following the work of Mok  \cite{Mk15} and the work of Xu (\cite{Xu16} and \cite{Xu17}, we  characterize the local $L$-packets  and the endoscopic classification of tempered automorphic representations  for the quasi-split unitary similitude groups.
 
Denote $L^2_{\disc,t}(\wt{\CG},\chi)$ to be the discrete tempered spectrum, whose irreducible sub-representations are tempered automorphic representations of $\wt{\CG}(\BA)$. 
Conjecturally, their Langlands functorial lift to automorphic representations of $\GL_n(\BA_E)$ are generic. 
Also denote by $\CA_{\disc, t}(\wt{\CG})$ for the subset of $\CA_\disc(\wt{\CG})$, whose elements are tempered automorphic representations.

In the rest of this paper,   we  take the quasi-split unitary similitude group $\GU_{n}$ as follows
\begin{equation}\label{eq:GU-n}
\GU_{n}=\{ g\in \Res_{E/F}\GL_{n}\mid {}^{t}\bar{g}J'_{n}g=\lam(g) J'_{n}\},
\end{equation}
where  $J'_n= \ppair{\begin{smallmatrix}
&J'_{n-1}\\(-1)^{n-1}&	
\end{smallmatrix}}$ and $\lam(g)\in \BG_{m}$ is the similitude factor of $g$.
For convenience, we choose the different Hermitian form in \eqref{eq:GU-n}instead of \eqref{eq:GU-2n}. 
And  $\RU_n$ is the quasi-split unitary subgroup of $\GU_n$ with the similitude 1.
 Then we have the following extension
\begin{equation}\label{eq:isogen-G}
1\longrightarrow \RU_n\longrightarrow \GU_n\stackrel{\lam}{\longrightarrow}D
\longrightarrow 1,
\end{equation} 
where $D$ is the quotient group isomorphic to $\BG_m$ in our case.
The exact sequence on the dual sides of \eqref{eq:isogen-G} is 
\begin{equation}\label{eq:daul-exact}
1\longrightarrow \GL_1(\BC)\longrightarrow \widehat{\GU}_n\stackrel{{\bf p}}{\longrightarrow} \widehat{\RU}_n \longrightarrow 1,	
\end{equation}
which is split.
Here $\widehat{\GU}_n=\GL_n(\BC)\times \GL_1(\BC)$ and $\widehat{\RU}_n$ are the complex dual groups of $\GU_n$ and $\RU_n$ respectively,  and ${\bf p}$ is the projection of the first coordinate.

\subsection{Endoscopic classification of $\RU_n$}\label{sec:U}
In this section, we will recall the endoscopic classification of quasi-split unitary groups in \cite{Mk15}
and follow the notation in \cite[Chapter 2]{Mk15}.

Define if $F$ is a local field
$$
\CZ_E=\cpair{\chi\colon E^\times\to \BC^\times \text{ unitary, } \chi\chi^c=1},
$$
and if $F$ is a number field
$$
\CZ_E=\cpair{\chi\colon \BA_E^\times/E^\times\to \BC^\times \text{ unitary, } \chi\chi^c=1}
$$
where $\chi^c(x):= \chi(\bar{x})$. 
Over the number field,
we have a partition $\CZ_E=\CZ_E^+\sqcup \CZ^-_E$,
$$
\CZ_E^+=\{\chi\in\CZ_E\colon \chi\vert_{\BA^\times_F}=1\} \text{ and }
\CZ_E^-=\{\chi\in\CZ_E\colon \chi\vert_{\BA^\times_F}=\delta_{E/F}\},
$$  
where $\delta_{E/F}$ is the quadratic character associated to $E/F$ by the global class field theory. 
Similarly, one has a partition in the local case. 
Each character $\chi_\kappa\in\CZ_E^\kappa$ for $\kappa=\pm 1$ may be also identified as a character of the Weil group $W_E$ of $E$.

Denote $G_{E/F}(n)$ to be the algebraic group $\Res_{E/F}\GL_{n}$ over $F$.
Let $\theta$ be the   $F$-automorphism of $G_{E/F}(n)$ defined by
$$
\theta(g)=J'_n{}^t\bar{g}^{-1}J'^{-1}_n,
$$ 
which fixes the standard choice of Whittaker datum $(B_{E/F},\Lam)$.
Here $B_{E/F}$ is the Borel subgroup consisting of upper triangular matrix and $\Lam$ is a non-degenerate character of its unipotent radical.
By the restriction, $(B,\Lam)$ is a Whittaker datum of $\RU_n$.
We use those Whittaker data to normalize the local Langlands correspondence of $\RU_n$.

For $\kappa=\pm 1$ and $\chi_\kappa\in\CZ^\kappa_E$, denote $\xi_\kappa\colon {}^L\RU_n\to {}^LG_{E/F}(n)$ to be the embedding of $L$-groups defined in \cite[(2.1.9)]{Mk15},
which will be recalled in Section \ref{sec:conjecture}.

Following \cite[Section 2.4]{Mk15}, denote $\wt{\CE}_{ell}(n)$ to be the set of equivalence classes of elliptic twisted endoscopic data of the twisted group $G_{E/F}(n)\rtimes\apair{\theta}$.
Each member in $\wt{\CE}_{ell}(n)$ is given by 
$$
(\CG',\xi')=(\RU_{N_1}\times \RU_{N_2},\xi_{\underline{\chi}_{\underline{\kappa}}})
$$
where $\xi_{\underline{\chi}_{\underline{\kappa}}}$ is an $L$-embedding from ${^L\CG'}$ to $ {^LG_{E/F}(N)}$ defined in \cite[(2.1.10)]{Mk15},
and the conditions on $\underline{\chi}_{\underline{\kappa}}$ are given in \cite[Page 17]{Mk15}.
For a simple twisted endoscopic datum $(\RU_{n}, \xi_{\chi_\kappa})$ of $G_{E/F}(n)$,
the sign $(-1)^{n-1}\cdot\kappa$ is called the {\sl parity} of the datum.

\subsubsection{Local $L$-packets} \label{sec:local-U}
In this section, we consider the local case, i.e., $F$ is a local field of characteristic 0.
In the local case, denote by $L_F$ the local Langlands group
$$
L_F:=\begin{cases}
	W_F  &\text{ if $F$ is archimedean}\\
	W_F\times \SL_2(\BC)  &\text{ if $F$ is non-archimedean.}
\end{cases}
$$
and the $L$-parameters for $\CG(F)$ are admissible homomorphisms form $L_F$ to its Langlands dual group ${}^L\CG$.

Denote $\Phi(\RU_n(F))$ to be the set of $\widehat{\RU}_n$-conjugacy classes of $L$-parameters of $\RU_n(F)$, or simply $\Phi(\RU_n)$.
For $\kappa=\pm 1$ and  $\chi_{\kappa}\in\CZ^\kappa_E$, one has a map
\begin{equation}\label{eq:xi}
\xi_{\chi_\kappa,*}\colon \Phi(\RU_n)\to\Phi(G_{E/F}(n)) \text{ via }
\phi\mapsto \xi_{\chi_\kappa}\circ\phi.	
\end{equation}

By \cite[Lemma 2.2.1]{Mk15}, we may identify $\Phi(\RU_n)$ with the set of parameters in $\Phi(G_{E/F}(n))$ that are conjugate self-dual with parity $\mu=(-1)^{n-1}\cdot \kappa$ (refer to \cite[Section 2.2]{Mk15} for instance).
Define  $\Phi_{bdd}(\RU_n)\subset \Phi(\RU_n)$ to be set of bounded parameters, that is, those parameters whose images of $W_F$ are bounded. 

For $\phi\in \Phi(\RU_n)$, let $S_\phi$ be the centralizer in $\widehat{\RU}_n$ of the image of $\phi$.
Define $\CS_\phi$ to be the component group of $S_\phi/Z(\widehat{\RU}_n)^{\Gamma_F}$, where $\Gamma_F$ is the absolute Galois group.

Denote by $\Pi(\RU_n(F))$  the set of isomorphism classes of irreducible admissible representations of $\RU_n(F)$,
and by $\Pi_{temp}(\RU_n(F))$ the subset of irreducible tempered representation of $\RU_n(F)$.
\begin{thm}[{\cite[Theorem 2.5.1]{Mk15}}]\label{thm:unitary:local}
Fix the choice of Whittaker datum $(B,\Lam)$ of $\RU_n$.
For $\phi\in\Phi_{bdd}(\RU_n(F))$, there exists a finite set $\Pi_\phi$ of $\Pi(\RU_n(F))$ equipped with a canonical mapping 
$$
\pi\in \Pi_\phi\mapsto \apair{\cdot,\pi}
$$
from $\Pi_\phi$ to the character group $\widehat{\CS}_\phi$ of $\CS_\phi$, such that it sends the $(B,\Lam)$-generic representation to the trivial character. 
If $F$ is non-archimedean, the map is bijective.

Furthermore, $\Pi_{temp}(\RU_n(F))$ is a disjoint union of the packets $\Pi_\phi$ for all $\phi\in\Phi_{bdd}(\RU_n(F))$.
\end{thm}

Next, we recall the definition of generic $L$-parameters and explicate the structure of generic $L$-packets.
For each $\phi\in \Phi(\RU_n(F))$, there exists a datum $(M,\phi_M,\lam)$ such that
\begin{itemize}
	\item $\phi_M$ is a local tempered $L$-parameter of some Levi subgroup $M(F)$ of $\RU_n$, which is isomorphic to a product of general linear groups and a quasi-split unitary groups;
	\item $\lam$ is a unramified character of  $M(F)$ in the open chamber of the standard parabolic subgroup $P$ containing $M$;
	\item $\phi$ is the composition of $\phi_M\otimes\lam$ under the $L$-embedding ${}^L M\to{}^L\RU_n$.
\end{itemize}
By the local Langlands corresponds for general linear groups and Theorem \ref{thm:unitary:local}, one has the local $L$-packet $\Pi_{\phi_M}$ and the pairing $\apair{\cdot,\pi_M}$ for $\pi_M\in \Pi_{\phi_M}$,
which is defined trivially on the general linear groups part of $M$. 

An $L$-parameter $\phi$ is called {\it generic} if the parabolic induction $\CI_P(\pi_{M,\lam}):=\Ind^{\RU_n(F)}_{P(F)}\pi_M\otimes\lam$ for the generic member $\pi_M\in \Pi_{\phi_M}$  is irreducible.
Denote $\Phi_{gen}(\RU_n(F))$ to be the subset of $\Phi(\RU_n(F))$ consisting of all generic $L$-parameters.

For $\phi\in \Phi_{gen}(\RU_n(F))$,
define the {\it local generic $L$-packet} $\Pi_{\phi}$ by
\begin{equation}\label{eq:local-generic-packet}
\Pi_{\phi}=\{\pi=\CI_P(\pi_{M,\lam})\colon \pi_M\in\Pi_{\phi_M}\},	
\end{equation}
and the pairing $\apair{\cdot,\pi}$ for $\pi\in \Pi_\phi$ by, for $s\in\CS_\phi=\CS_{\phi_M}$,
$$
\apair{s,\pi}=\apair{s,\pi_M}
$$
where $\pi_M$ is given by $\pi=\CI_P(\pi_{M,\lam})$.

Remark that this equivalent definitions of generic $L$-parameters and $L$-packets are implied by the irreducibility of the standard module of generic representations for $\RU_n(F)$.
In this case, for all $\pi_M\in \Pi_\phi$, $\CI_P(\pi_{M,\lam})$ is irreducible and there is no ambiguity in the definition \eqref{eq:local-generic-packet}.

\subsubsection{Global generic Arthur packets} \label{sec:global-U}
In this section, we consider the global case, in which case $F$ is a number field. 
Here we only recall ${\Phi}_2(\RU_n,\xi_{\chi_\kappa})$,  the set of tempered discrete global $L$-parameters, 
instead of ${\Psi}_2(\RU_n,\xi_{\chi_\kappa})$,  the set of all discrete global Arthur-parameters in  \cite{Mk15}.

Denote $\wt{\Phi}_{\simp}(n)$ to be the set of conjugate self-dual simple generic parameters for $G_{E/F}(n)$.
The global $L$-parameters in $\wt{\Phi}_\simp(n)$ correspond to the equivalence classes of the conjugate self-dual, irreducible cuspidal automorphic representations of $G_{E/F}(n,\BA_F)$.
Here an irreducible cuspidal automorphic representation $\tau$ of $G_{E/F}(n)$ is called {\sl conjugate self-dual} if $\tau\cong\tau^*$, where
$\tau^*= (\tau^c)^\vee$ the contragredient of the conjugate $\tau^c$.

Suppose that $\phi^n\in\wt{\Phi}_\simp(n)$.
By \cite[Theorem 2.5.4]{Mk15}, the (partial) Asai $L$-function
$L(s,\phi^n,Asai^{\eta})$ has a (simple) pole at $s=1$ with the sign $\eta=\pm 1$.
We say that $\phi^n$ is   {\it conjugate orthogonal} if $\eta=1$ and {\it conjugate symplectic} if $\eta=-1$.
Here
$L^S(s,(\tau,1),Asai^+)$ is the (partial) Asai $L$-function of $\tau$ and $L^S(s,(\tau,1),Asai^-)$ is the (partial) Asai $L$-function of
$\tau\otimes\delta_{E/F}$.

Denote $\wt{\Phi}(n)$ to be the set of conjugate self-dual generic parameters, namely, consisting of  elements of form
\begin{equation}\label{ellgap}
\phi^n=\phi_1^{n_1}\boxplus\cdots\boxplus\phi_r^{n_r}
\end{equation}
with $n=\sum_{i=1}^rn_i$ and $\phi^n=(\phi^n)^*$. 

Fix a sign $\kappa=\pm 1$ and $\chi_\kappa\in\CZ_E^\kappa$.
Referring to \cite[Definition 2.4.7]{Mk15}, define the set $\Phi_2(\RU_n,\xi_{\chi_\kappa})$ of {\it the generic Arthur parameters} to be set  of $\phi=(\phi^n,\breve{\phi})$ such that $\phi^n\in \wt{\Phi}(n)$ is of form \eqref{ellgap}
for mutually distinct $\phi^{n_i}\in\wt{\Phi}(n_i)$ satisfying $\kappa_i(-1)^{n_i-1}=\kappa(-1)^{n-1}$ for all $1\leq i\leq r$,
and $\breve{\phi}$ is an $L$-homomorphism from $L_{\phi^n}\times\SL_2(\BC)$ to ${}^L\RU_n$ determined by $\phi^n$ and $\xi_{\chi_\kappa}$ given in \cite[Definition 2.4.5]{Mk15}.
Here $L_{\phi^n}\times\SL_2(\BC)$ is a substitute  for the conjectural global Langlands group.
We may also identify $\Phi_2(\RU_n,\xi_{\chi_\kappa})$ to be a subset of $\wt{\Phi}(n)$ via $\phi\mapsto \phi^n$. 

For $\phi=(\phi^n,\breve{\phi})\in \Phi_2(\RU_n,\xi_{\chi_\kappa})$,
similar to the local case, define $\CS_\phi$ to be the component group of $\ov{S}_\phi$ where $S_\phi=\text{Cent}(\Im\breve{\phi},\widehat{\RU}_n)$ and $\ov{S}_\phi=S_\phi/Z(\widehat{\RU}_n)^{\Gamma_F}$.
In addition, over a local place $\nu$ of $F$, let  $\phi^n_\nu$ be the localization of $\phi^n$.
Due to \cite[Corollary 2.4.11]{Mk15}, the local $L$-parameter $\phi^n_\nu$ factors through $\xi_{\chi_\kappa,\nu}$, which is in $\Phi^+_{unit}(\GL_n(E_\nu))$.
Hence one has the localization $\phi_\nu\in\Phi_{gen}(\RU_n(F_\nu))$ such that $\phi^n_\nu=\xi_{\chi_\kappa,\nu}\circ \phi_\nu$,
and the localization maps $\CS_{\phi}\to\CS_{\phi_\nu}$, referring to \cite[(2.4.23)]{Mk15}.
More precisely, $\phi_\nu$ is in $ \Phi^+_{unit}(\RU_n(F_\nu))\subset \Phi_{gen}(\RU_n(F_\nu))$ (cf. Proposition 3.10 in \cite{Xu17}).
The unexplained notation and more details of the localization can be found in Section 2.4 of \cite{Mk15}.

Define the global packet $\Pi_\phi$ associated to $\phi$ as the restricted tensor product of  the local generic $L$-packets $\Pi_{\phi_\nu}$ (see \eqref{eq:local-generic-packet}):
$$
\Pi_\phi=\otimes'_\nu\Pi_{\phi_\nu}=\{\pi=\otimes'_\nu\pi_\nu,\pi_\nu\in\Pi_{\phi_\nu},\apair{\cdot,\pi_\nu}=1 \text{ for almost all $\nu$}\},
$$
and the pairing of $\Pi_\phi$ and $\CS_\phi$ to be $\apair{x,\pi}=\prod_\nu\apair{x_\nu,\pi_\nu}$ for $x\in\CS_\phi$ via the localization map $\CS_{\phi}\to\CS_{\phi_\nu}$.
Then we recall the endoscopic classification of tempered discrete automorphic representations for quasi-split unitary groups established in \cite{Mk15}.
\begin{thm}[{\cite[Theorem 2.5.2]{Mk15}}]\label{thm:Mok}
Fix $\kappa=\{\pm 1\}$ and $\chi_{\kappa}\in\CZ^\kappa_E$.
There is a $\CH(G)$-module decomposition 
$$
L^2_{disc,t}(\RU_n(F)\bks \RU_n(\BA))=
\bigoplus_{\phi\in\Phi_2(\RU_n,\xi_{\chi_\kappa})}\bigoplus_{\pi\in\Pi_{\phi},\apair{\cdot,\pi}=1}\pi,
$$
where $\CH(G)$ is the global Hecke algebra of $\RU_n(\BA_F)$.
\end{thm}

\subsection{Endoscopic classification of $\GU_n$}\label{sec:endoscopic-GU}

In this section, following Section \ref{sec:U} and Xu's results of \cite{Xu16} and \cite{Xu17} in $\GSp_{2n}$ and $\GO_{2n}$, 
we give a conjectural endoscopic classification of $\GU_n$ by restricting the representations of $\GU_n$ into $\RU_n$.

\subsubsection{Local $L$-packets}\label{sec:local-GU}
Firs, let us characterize the local $L$-packets of  $\GU_{n}$ which satisfies the conjectural endoscopic character identities.
Suppose that $F$ be a local field of characteristic 0.   

Define $\Phi(\GU_n)$ to be the local Langlands parameters of $\GU_n(F)$.
Due to Labesse, every $L$-parameter $\phi$ of $\RU_n(F)$ can be lifted to a Langlands parameter $\wt{\phi}$ of $\GU_n(F)$ via
\begin{equation}\label{eq:diagram-projection}
\xymatrix{
	L_F\ar[r]^{\wt{\phi}}\ar[dr]_{\phi}&{}^L\GU_n\ar[d]^{{\bf p}}\\
	&{}^L\RU_n.
}
\end{equation}
For $\phi\in\Phi(\RU_n)$ and a lift $\wt{\phi}$  of $\phi$,
similarly define $S_{\wt{\phi}}={\rm Cent}(\Im\wt{\phi},\widehat{\GU_n})$ and $\CS_{\wt{\phi}}$ to be the component group of $S_{\wt{\phi}}/Z(\widehat{\GU_n})^{\Gamma_F}$.
Referring to  \cite[Section 3.2]{Xu16},
the dual exact sequence \eqref{eq:daul-exact} induces the long exact sequence (\cite[(3.3)]{Xu16})
$$
1\longrightarrow  S_{\wt{\phi}}/D^{\Gamma_F}\longrightarrow S_{\phi}\stackrel{\delta}{\longrightarrow} H^1(W_F,\hat{D}),
$$
and further induces the following exact sequence (\cite[(3.6)]{Xu16})
$$
1\longrightarrow \CS_{\wt{\phi}}\longrightarrow \CS_{\phi}\stackrel{\alpha}{\longrightarrow} \Hom(\GU_n(F)/\RU_n(F),\BC^\times),
$$
where $\alpha$ is the composition of $\delta$ and the isomorphism 
$$H^1(W_F,\hat{D})\to \Hom(\GU_n(F)/G\RU_n(F),\BC^\times).$$
By \cite[Lemma 3.4]{Xu16}, the image $\alpha(\CS_\phi)$ is contained in 
$$
X:=\Hom(\GU_n(F)/Z_{\GU_n}(F)G\RU_n(F),\BC^\times).
$$

Note that
if $n$ is even,
$$
\GU_n(F)/Z_{\GU_n}(F)\RU_n(F)\cong \lam(\GU_n(F))/\lam(Z_{\GU_n}(F))
=F^\times/{\rm Nm}{E}^{\times}\cong \mu_2;
$$
If $n$ is odd, then $\GU_n(F)=Z_{\GU_n}(F)\RU_n(F)$.
Thus $X=\{1,\delta_{E/F}\}$ when $n$ is even and $X=\{1\}$ when $n$ is odd;
and $\alpha$ is either a trivial map or a surjective map.
Let $\wt{\pi}$ be an irreducible smooth admissible representation of $\GU_n(F)$.
Define $X(\wt{\pi})=\{\omega\in X\colon \wt{\pi}\otimes\omega\cong \wt{\pi}\}$.
Following \cite[Corollary 6.7]{Xu16}, for instance, we obtain
\begin{lm}\label{lm:restriction}
Let $\wt{\pi}\in \Pi(\GU_n(F))$ and $\wt{\pi}\vert_{\RU_n(F)}$ be the restriction of $\wt{\pi}$ into $\RU_n(F)$. Then $\wt{\pi}\vert_{\RU_n(F)}$ is multiplicity-free.

Moreover,  $\wt{\pi}\vert_{\RU_n(F)}$ is irreducible if and only if $n$ is odd or $\wt{\pi}\ncong \wt{\pi}\otimes \delta_{E/F}$;
If $\wt{\pi}\cong \wt{\pi}\otimes \delta_{E/F}$ and $n$ is even, then $\wt{\pi}\vert_{\RU_n(F)}=\pi\oplus\pi\circ\Ad(g)$,
where $g=\diag\{aI_{n/2},I_{n/2}\}$ and $a\notin {\rm Nm}E$. 
\end{lm}

Let $\phi\in \Phi_{bdd}(\RU_n)$, and $\wt{\zeta}$ be a character of $Z_{\GU_n}(F)$ whose restriction to $Z_{\RU_n}(F)$ is the central character of $\Pi_{\phi}$.
Define $\wt{\Pi}_{\phi,\wt{\zeta}}$ to be the subset of $\Pi(\GU_n(F))$ with central character  $\wt{\zeta}$ whose restriction to $\RU_n(F)$ is contained in $\Pi_\phi$,
It is called {\it coarse $L$-packet} for $\GU_n$ defined in  \cite[Section 6.2]{Xu16}.

Due to \cite{Mk15}, $\CS_\phi$ is abelian. By Lemma \ref{lm:restriction}, the restriction of $\wt{\pi}$ to $\RU_n(F)$ is multiplicity-free.
Then Hypothesis 1 in \cite[Section 6.2]{Xu16} holds for unitary similitude groups
and  $X(\wt{\pi})=\alpha(\CS_\phi)$ by \cite[Proposition 6.13]{Xu16}.
Following from \cite[Proposition 6.14]{Xu16}, the coarse $L$-packets $\wt{\Pi}_{\phi,\wt{\zeta}}$ have the following properties
\begin{pro}\label{prop:coarse-unitary}
Suppose that $\phi\in \Phi_{bdd}(\RU_n)$ and $\wt{\zeta}$ is chosen as above.
\begin{enumerate}
	\item The orbits in $\Pi_\phi$ under the conjugate action of $\GU_n(F)$ all have size $\CS_\phi/\CS_{\wt{\phi}}$. If $F$ is nonarchimedean, there are exactly $|\CS_{\wt{\phi}}|$ orbits.
	\item There is a natural fibration 
	$$
	X/\alpha(\CS_\phi)\longrightarrow \wt{\Pi}_{\phi,\wt{\zeta}}\stackrel{\Res}{\longrightarrow}\Pi_\phi/\GU_n(F),
	$$
	where $\Res$ is the restriction of the representations of $\GU_n(F)$.
	\item \label{prop:part:coarse-pair} There is a unique pairing $\wt{\pi}\to \apair{\cdot, \wt{\pi}}_\phi$ from $\wt{\Pi}_{\phi,\wt{\zeta}}/X$ into $\widehat{\CS_{\wt{\phi}}}$ satisfying 
	$$
	\apair{x,\wt{\pi}}_\phi=\apair{\iota(x),\pi}_{\phi},
	$$
	where $\iota\colon \CS_{\wt{\phi}}\to \CS_\phi$ and $\pi$ is in the restriction of $\wt{\pi}$. 
	It sends the generic representation to the trivial character.
	Moreover, this map from $\wt{\Pi}_{\phi,\wt{\zeta}}/X$ to $\widehat{\CS}_{\wt{\phi}}$ is injective; and when $F$ is nonarchimedean, it is a bijection.
\end{enumerate}
\end{pro}
More precisely, the twisting of $X$ on $\Pi(\GU_n(F))$ is given by $\pi\mapsto \pi\otimes\omega$ for $\omega\in X$.
By definition, $\wt{\Pi}_{\phi,\zeta}$ is stable under this twisting.

\begin{rmk}
Proposition \ref{prop:coarse-unitary} is an extension of  Proposition 6.29 in \cite{Xu16} to quasi-split unitary similitude groups.
\end{rmk}

Referring to Conjecture 6.18 in \cite{Xu16}, one expects the following conjectural refinement of $L$-packets of $\GU_n$, which gives a section of the pairing of $\wt{\Pi}_{\phi,\wt{\zeta}}$ and $\CS_{\wt{\phi}}$ in Part \eqref{prop:part:coarse-pair} of Proposition \ref{prop:coarse-unitary}.

\begin{conj} \label{conj:refine-L-packet}
Suppose that $\phi\in \Phi_{bdd}(G)$ and $\wt{\zeta}$ is chosen as above.

Then there exists a subset $\Pi_{\wt{\phi}}$ of  $\wt{\Pi}_{\phi,\wt{\zeta}}$ unique up to twisting by $X$, and it is characterized by the following properties:
\begin{enumerate}
	\item \label{thm:twist-endoscopy:L-pakcet} $\wt{\Pi}_{\phi,\wt{\zeta}}=\bigsqcup_{\omega\in X/\alpha(\CS_\phi)}\Pi_{\wt{\phi}}\otimes\omega$.
	\item For $\wt{f}\in C^\infty_c(\GU_n(F),\wt{\zeta}^{-1})$, the distribution 
	$$
	\tilde{f}(\wt{\phi}):=\sum_{\wt{\pi}\in\Pi_{\wt{\phi}}}\tilde{f}_{\GU_n}(\wt{\pi})
	$$
	is stable, where $C^\infty_c(\GU_n(F),\wt{\zeta}^{-1})$ is the space of $\wt{\zeta}$-equivariant smooth functions on $\GU_n(F)$ with compact support modulo $Z_{\GU_n}(F)$ and
	$$
	\tilde{f}_{\GU_n}(\wt{\pi})=\tr\pi(\tilde{f})=\tr\int_{Z_{\GU_n}(F)\bks \GU_n(F)} \wt{f}(x)\pi(x)\ud x
	$$
	is the character distribution of $\wt{\pi}$.
\end{enumerate}
\end{conj}

\begin{rmk}
Following Arthur's endoscopic classification in \cite{A13}, 
Xu \cite{Xu17} characterized the $L$-packets of symplectic and special orthogonal similitude groups in Theorem 4.6 of \cite{Xu17}, 
by using the stabilization of the twisted Arthur-Selberg trace formula due to M\oe glin and Waldspurger.  
For unitary similitude groups, following Mok's work in \cite{Mk15}, applying an argument similar to Theorem 4.6 of \cite{Xu17},
we expect that  Conjecture \ref{conj:refine-L-packet} with 
the conjectural endoscopic character identities (see Parts (3) and (4) of \cite[Conjecture 6.18]{Xu16}) holds.
\end{rmk}

\begin{rmk}
In general, one may define the refined local $L$-packets for all Langlands parameters of $\GU_n(F)$.
By Langlands' classification,
similar to Part (4) of \cite[Remark 4.7]{Xu16} and the generic $L$-packet in Section \ref{sec:local-U}, for $\phi\in \Phi(\RU_n)$, since $\Pi_\phi=\CI_P(\Pi_{\phi_{M,\lam}})$,
we may define $\Pi_{\wt{\phi}}=\CI_{\wt{P}}(\Pi_{\wt{\phi}_{M,\lam}})$.
Here  $\wt{P}$ is the parabolic subgroup of $\GU_n$ whose restriction to $\RU_n$ is $P$ and  $\CI_{\wt{P}}$ is the corresponding parabolic induction.
\end{rmk}

\subsubsection{Global generic Arthur packets}


In the global case,  we use tempered Arthur parameters $\phi$ in $\Phi_2(\RU_n,\xi_{\chi_\kappa})$ to construct the tempered Arthur packet $\Pi_{\wt{\phi}}$ of $\GU_n(\BA)$.

Recall from Section \ref{sec:global-U}, we have the substitute $L_{\phi^n}\times \SL_2(\BC)$ for the global Langlands group,
and $\breve{\phi}$ is an $L$-homomorphism from $L_{\phi^n}\times\SL_2(\BC)$ to ${}^L\RU_n$ determined by $\phi^n$ and $\xi_{\chi_\kappa}$.
Given a central character $\wt{\zeta}$,
one has a local lifting $\wt{\phi}_\nu$ of the localization $\phi_\nu$ over all local places $\nu$.

Following Section 3.3 in \cite{Xu17}, we may use the group cohomology of $\breve{\phi}$ to define $S_{\wt{\phi}}$.
Then denote $\CS_{\wt{\phi}}$ to be the component group of $S_{\wt{\phi}}/Z(\widehat{\GU_n})^{\Gamma_F}$.
Following \cite[Section 2]{Xu17} we also have an exact sequence
\[
1\longrightarrow \CS_{\wt{\phi}}\longrightarrow \CS_{\phi}\stackrel{\alpha}{\longrightarrow} \Hom(\GU_n(\BA)/\GU_n(F)\RU_n(\BA),\BC^\times).	
\]
Let  $Y$ be the set of characters of $\GU_n(\BA)/\GU_n(F)Z_{\GU_n}(\BA)\RU_n(\BA)$. 
By Lemma 1.1.3 in \cite{CHL11}, $\GU_n(\BA)/\GU_n(F)Z_{\GU_n}(\BA)\RU_n(\BA)$ 
is trivial if $n$ is odd;
and it is an elementary abelian 2-group if $n$ is even.
Moreover, $\alpha(\CS_{\wt{\phi}})$ lies in $Y$.

\begin{conj}[{ \cite[Conjecture 5.16]{Xu17}}]\label{conj:endoscopy-GU}
Fix $\kappa=\{\pm 1\}$ and $\chi_{\kappa}\in\CZ^\kappa_E$.
\begin{enumerate}
	\item For $\phi\in\Phi_2(\RU_n,\xi_{\chi_\kappa})$, one can associate a global packet $\Pi_{\wt{\phi}}$ of $\CH(\GU_n)$-modules of irreducible admissible representations for $\GU_n(\BA)$ satisfying the following properties:
	\begin{enumerate}
		\item $\Pi_{\wt{\phi}}=\otimes'_\nu\Pi_{\wt{\phi}_\nu}$ where $\Pi_{\wt{\phi}_{\nu}}$ is some lift of $\Pi_{\phi_\nu}$ defined in Conjecture~\ref{conj:refine-L-packet} and Proposition \ref{prop:coarse-unitary}.
		\item there exists $\wt{\pi}\in\Pi_{\wt{\phi}}$ such that $\wt{\pi}$ is isomorphic to an automorphic representation as $\CH(\GU_n)$-modules.
	\end{enumerate}
	Moreover,  up to twisting by  characters of
	 $\GU_n(\BA)/\GU_n(F)\RU_n(\BA)$, $\Pi_{\wt{\phi}}$ is unique. And we can define a global character of $\CS_{\wt{\phi}}$ by
	$$
	\apair{x,\wt{\pi}}:=\prod_v\apair{x_v,\wt{\pi}_v} 
	\text{ for $\wt{\pi}\in \Pi_{\wt{\phi}}$ and $x\in \CS_{\wt{\phi}}$}. 
	$$
	\item 
There is a $\CH(\GU_n)$-module decomposition 
$$
L^2_{disc,t}(G(F)\bks G(\BA),\wt{\zeta})=
\bigoplus_{\phi\in\Phi_2(\RU_n,\xi_{\chi_\kappa})}
\bigoplus_{\omega\in Y/\alpha(\CS_\phi)}
\bigoplus_{\substack{\wt{\pi}\in\Pi_{\wt{\phi}}\otimes\omega\\
\apair{\cdot,\wt{\pi}}=1}}\wt{\pi}.
$$
\end{enumerate}
\end{conj}

\begin{rmk}
The global base change has been discussed in many literatures such as \cite[Corollary 5.3]{L11} and \cite[Prop 8.5.3]{M10}.
\end{rmk}

If we assume Conjecture \ref{conj:endoscopy-GU}, we may obtain the base change for all cuspidal automorphic representations $\pi$ of $\GU_n(\BA)$ in the tempered discrete spectrum.
The global Arthur packet $\Pi_{\wt{\phi}}$ contains a unique generic member,
and  all representations in $\Pi_{\wt{\phi}}$ are nearly equivalent to this generic automorphic representation.
Then we have the following corollary
\begin{cor}
Assume Conjecture \ref{conj:endoscopy-GU}. 
Let $\pi$ be an irreducible cuspidal automorphic representation of $\GU_6(\BA)$ in the tempered discrete spectrum. 

Then
$L^S(s,\pi\otimes\chi,\e{3})$ is entire unless $\omega_\pi\chi^2$ is a nontrivial quadratic character of $\GL_1(F)\bks \GL_1(\BA)$,
in which case $L^S(s,\pi\otimes\chi,\e{3})$ can have at most a simple pole at $s=0$ or $s=1$.
\end{cor}

\section{Unitary automorphic induction and poles}\label{sec:conjecture}
In this section, we will introduce the automorphic induction from $\GU^\circ_{K/E^+}(n)$ to $\GU_{2n}$, 
and show that such automorphic induction exists under the weak Langlands functoriality.
Then using this automorphic induction, we give a conjectural criterion on the existence of pole of   $L^S(s,\pi,\e{3}\otimes\chi)$.
 
\subsection{$L$-homomorphism}
Let $E^+/F$ be a quadratic extension over $F$.
Define $K=E^+\otimes_F E$  which is an \'etale algebra over $F$ of degree 4.
More precisely, $K=E\times E$ if $E^+\simeq E$, and $K$ is a quartic extension over $E$ if $E^+$ is not isomorphic to $E$.

If $E^+\not\simeq E$, then $K$ is a Galois field whose Galois group is the Klein 4-group.
Denote $L$ to be the subfield of $K$ which is the quadratic extension of $F$ distinct from $E$ and $E^+$.
Denote $\vartheta_{E'}$ to be the non-trivial elements in $\Gal(K/E')$, where $E'=E$, $E^+$ or  $L$.
For $E'=E^+$ or $L$, we have the canonical isomorphism from $\Gal(K/E')$ to $\Gal(E/F)$ given by $\vartheta\mapsto \vartheta_{E'}\vert_E=\iota_{E/F}$ (defined in Page \ref{pg:iota-E}).
So $\vartheta_{E^+},\vartheta_{L}\in W_F\smallsetminus W_E$,
and $\vartheta_{E^+}$ is the natural extension of $\iota_{E/F}$ from $E$ to $E^+\otimes_F E$.
Remark that when $E^+\simeq E$ we only consider the Galois group $\Gal(E/F)$, whose nontrivial element is denoted by $\vartheta_{E^+}$ for consistency.

Define $\GU^\circ_{K/E^+}(n)$ to be the subgroup of $\Res_{E^+/F}\GU_n$ given by
$$
\GU^\circ_{K/E^+}(n)=\{g\in\Res_{K/F}\GL_n\colon {}^tg\cdot J'_n\cdot \vartheta_{E^+}(g)=\lam(g)J'_n,~\lam(g)\in\BG_m\}
$$
where $\BG_m$ is the multiplicative group over $F$.
In particular, if $E^+\simeq E$, then $\GU^\circ_{K/E^+}(n)$ is isomorphic to 
$$
G(\RU_n \times \RU_{n})=\{(g,h)\in\GU_n\times \GU_{n}\colon \lam(g)=\lam(h)\}.
$$
Referring to \cite[Section 2.3]{M10},
the complex dual group $\widehat{\GU}^{\circ}_{K/E^+}(n)$  is $\GL_n(\BC)\times \GL_n(\BC)\times \GL_1(\BC)$, and its $L$-group is given by
$$
{}^L \GU^\circ_{K/E^+}(n)=(\GL_n(\BC)\times \GL_n(\BC)\times \GL_1(\BC))\rtimes \Gal(K/F).
$$
The actions of $\vartheta_{E^+}$ and $\vartheta_{E}$ in $\Gal(K/F)$ are  given by
\begin{align*}
 & \vartheta_{E^+}\colon (g,h,a)\mapsto (J'_n {}^tg^{-1} J'^{-1}_n,J'_n {}^t h^{-1} J'^{-1}_n,a\det(g)\det(h)),\\ 
 & \vartheta_E\colon (g,h,a)\mapsto (h,g,a).	
\end{align*}
Remark that the action of $\vartheta_E$ is trivial if $E^+\simeq E$.

If $E^+ \simeq E$, the group $\Res_{E/F}\GU^\circ_{K/E^+}(n)$ is isomorphic to $\Res_{K/F}\GL_n\times \Res_{E/F}\GL_1$ and its complex dual group is
$$
\GL_n(\BC)\times\GL_n(\BC)\times\GL_n(\BC)\times\GL_n(\BC)\times\GL_1(\BC)\times\GL_1(\BC),
$$
and
the action  of $\vartheta_{E^+}$ on its complex dual group $\GL^{\times 4}_n(\BC)\times\GL^{\times 2}_1(\BC)$ is given by
$$
\vartheta_{E^+}\colon (g_1,g_2,h_1,h_2,a,b)\mapsto (h_1,h_2,g_1,g_2,b,a)
$$
where $g_1,g_2,h_1,h_2\in\GL_n(\BC)$ and $a,b\in\GL_1(\BC)$.
If $E^+\not\simeq E$, then the action of $\vartheta_{E}$ is given by
$$
\vartheta_{E}\colon (g_1,g_2,h_1,h_2,a,b)\mapsto (g_2,g_1,h_2,h_1,a,b).
$$

Let  ${}^L b_{E/F}$ and ${}^L i_{E^+/F}$ be the restriction and induction $L$-embeddings of $L$-groups respectively.
According to the principle of Langlands functoriality, the homomorphisms ${}^L b_{E/F}$ and ${}^L i_{E^+/F}$ between $L$-groups will be reflected by correspondences between automorphic representations, which are {\it Base change} and {\it Automorphic induction}.

Let $\xi$ be a character of $\BA^\times_K$ and we identity it as a character of $W_K$ via the global class field theory.
Suppose that $\xi$ satisfies the following conditions: for $w\in W_K$
\begin{equation}\label{eq:xi}
\begin{array}{ll}
\xi(\vartheta_{E^+}\cdot w\cdot \vartheta_{E^+}^{-1})=\xi(w)^{-1}
&\xi(\vartheta_{E^+}^2)=(-1)^n\\
\xi(\vartheta_E\cdot w\cdot \vartheta_E^{-1})=\xi(w) &
\xi(\vartheta^2_{E})=(-1)^n \quad \text{(for $E^+\not\simeq E$)}.
\end{array}
\end{equation}
Equivalently,  $\xi$ as a character of $\BA_K^\times$, via the global class field theory, satisfies
\begin{align}
&\xi\circ\vartheta_{E^+}=\xi^{-1} && \xi\vert_{\BA_{E^+}^\times}=\delta^n_{E^+/F}\\
&\xi\circ\vartheta_{L}=\xi^{-1} && \xi\vert_{\BA_{L}^\times}=1	\quad \text{(for $E^+\not\simeq E$)}. \label{eq:xi-L}
\end{align}
Note that for $E^+\not\simeq E$ we give an equivalent condition \eqref{eq:xi-L} to the one in \eqref{eq:xi} for a shorter expression.

\subsubsection{Base change}
More precisely, for the group ${}^L\GU_{E/F}(n)$, define 
$$
{}^L b_{E/F}\colon {}^L\GU_{E/F}(n) \to {}^L \Res_{E/F}(\GL_n\times \GL_1)
$$
given by: for $g\in \GL_n(\BC)$ and $w\in W_E$,
\begin{align*}
& (g,a)\rtimes 1\mapsto (g,{}^tg^{-1}, a, a\det(g))\rtimes 1  \\
&(I_n,1)\rtimes w\mapsto (I_n,I_n,1,1)\rtimes w\\
&(I_n,1)\rtimes \iota_{E/F} \mapsto (J'_n,J'^{-1}_n,1,1)\rtimes \iota_{E/F},
\end{align*}
where $\iota_{E/F}$ is considered as  a preimage in $W_F\smallsetminus W_E$.
(The definition is independent on the choice of the preimage of $\iota_{E/F}$.)

For the group ${}^L\GU^\circ_{K/E^+}(n)$, define 
$$
{}^L b_{E/F,\xi}\colon {}^L\GU^\circ_{K/E^+}(n)  \to {}^L(\Res_{K/F}\GL_n\times \Res_{E/F}\GL_1)
$$
given by: for $g,h\in \GL_n(\BC)$ and $w\in W_K$,
\begin{align*}
& (g,h,a)\rtimes 1\mapsto (g,h,{}^tg^{-1},{}^th^{-1}, a, a\det(g)\det(h))\rtimes 1  \\
&(I_n,I_n,1)\rtimes w\mapsto (\xi(w)I_n,\xi(w)I_n,\xi(w)^{-1}I_n,\xi(w)^{-1}I_n,\xi(w)^{-n},\xi(w)^{n})\rtimes w\\
&(I_n,I_n,1)\rtimes \vartheta_{E^+} \mapsto ((-1)^nJ'_n,J'_n,J'^{-1}_n,(-1)^nJ'^{-1}_n,1,(-1)^n)\rtimes \vartheta_{E^+}.
\end{align*} 
When $E^+\not\simeq E$,  define
$$
(I_n,I_n,1)\rtimes \vartheta_E\mapsto (i^nI_n,i^nI_n,i^{-n}I_n,i^{-n}I_n,(-i)^{n^2},(-i)^{n^2})\rtimes \vartheta_E.
$$

For example, for generic cuspidal representations of  $\GU_n(\BA)$,
we have the following standard base change proved by Kim and  Krishnamurthy
\begin{lm}[{\cite[Lemma 5.4]{KK08}}]\label{lm:Base-Change}
Suppose that $\pi$ is a globally generic cuspidal automorphic representation of $\GU_n(\BA)$. 
Let $\pi'$ be a generic cuspidal constitute of $\RU_n(\BA)$.
Then $b_{E/F}(\pi)=b_{E/F}(\pi')\otimes\bar{\omega}_\pi$ is the stable base change lift of $\pi$,
where $\bar{\omega}_\pi:=\omega_\pi\circ\iota_{E/F}$.
\end{lm}

\subsubsection{Automorphic induction}

Define
$$
{}^L i_{E^+/F,\xi}\colon {}^L\GU^\circ_{K/E^+}(n)\to {}^L\GU_{E/F}(2n)
$$
given by: for $g,h\in \GL_n(\BC)$ and $w\in W_K$,
\begin{align*}
& (g,h,a)\rtimes 1 \mapsto (\ppair{ \begin{smallmatrix}
g&0\\0&h
\end{smallmatrix}},a)\rtimes 1  \\
&(I_n,I_n,1)\rtimes w\mapsto (\xi(w)I_{2n},\xi(w)^{-n})\rtimes w\\
&(I_n,I_n,1)\rtimes \vartheta_{E^+} \mapsto (\ppair{ \begin{smallmatrix}
I_{n}&0\\0&I_{n}
\end{smallmatrix}},1)\rtimes \vartheta_{E^+}.
\end{align*} 
When $E^+\not\simeq E$,  define
$$
(I_n,I_n,1)\rtimes \vartheta_E\mapsto (i^n\ppair{ \begin{smallmatrix}
0&I_n\\I_n&0
\end{smallmatrix}},(-i)^{n^2})\rtimes \vartheta_E,
$$
where $i=\sqrt{-1}$.

Define
$$
{}^L i_{E^+/F}\colon {}^L(\Res_{K/F}\GL_n\times \Res_{E/F}\GL_1)\to {}^L \Res_{E/F}(\GL_{2n}\times \GL_1)
$$
given by:
\begin{align*}
 & (g_1,g_2,h_1,h_2,a,b)\rtimes 1\mapsto (\ppair{ \begin{smallmatrix}g_1&0\\0&g_2\end{smallmatrix}},\ppair{ \begin{smallmatrix}h_1&0\\0&h_2\end{smallmatrix}},a,b)\rtimes 1\\
&(I_n,I_n,I_n,I_n,1,1)\rtimes w\mapsto (I_{2n},I_{2n},1,1)\rtimes w\\
&(I_n,I_n,I_n,I_n,1,1)\rtimes \vartheta_{E^+} \mapsto (I_{2n},I_{2n},1,1)\rtimes \vartheta_{E^+}.
\end{align*}
When $E^+\not\simeq E$,  define
$$
(I_n,I_n,I_n,I_n,1,1)\rtimes \vartheta_{E}\mapsto (i^n\ppair{ \begin{smallmatrix}
0&I_n\\I_n&0
\end{smallmatrix}},(-i)^n\ppair{ \begin{smallmatrix}
0&I_n\\I_n&0
\end{smallmatrix}},(-i)^{n^2},(-i)^{n^2})\rtimes 1.
$$

Following the definition,   it is easy to check the following diagram is commutative:
\begin{equation}\label{diagram}
\xymatrix{
	{}^L\GU^\circ_{K/E^+}(n) \ar[rrr]^{{}^L b_{E/F,\xi}}\ar[d]^{{}^L i_{E^+/F,\xi}}&&&  {}^L(\Res_{K/F}\GL_n\times \Res_{E/F}\GL_1) \ar[d]^{{}^L i_{E^+/F}}\\
	{}^L\GU_{E/F}(2n) \ar[rrr]^{{}^L b_{E/F}} &&&  {}^L \Res_{E/F}(\GL_{2n}\times \GL_1)
}
\end{equation}

Remark that when $E^+\simeq E$ the automorphic induction ${}^L i_{E^+/F,\xi}$ is the endoscopic lifting.

\subsection{Automorphic induction}\label{sec:a-i}
In this section, we will show the existence of the automorphic inductions under the weak Langlands functoriality associated to the $L$-homomorphism ${}^Li_{E^+/F,\xi}$.

\begin{thm}\label{thm:automorphic-induction}
Fix a choice of $\xi$ given in \eqref{eq:xi} and $K=E^+\otimes E$ be a quadratic extension of $E$.
Suppose that $\tau$ is an irreducible generic cuspidal automorphic representation of $\GU^\circ_{K/F}(n,\BA)$.

Then there exists an irreducible generic cuspidal automorphic representation $\pi$ of $\GU_{2n}(\BA)$, which is a functorial lift of $\tau$ corresponding to the $L$-homomorphism ${}^Li_{E^+/F,\xi}$ over almost all local places $\nu$.
\end{thm}
Such $\pi$ is called an {\it automorphic induction} of $\tau$, denoted by $i_{E^+/F,\xi}(\tau)$.

\begin{rmk}
Given $\xi$, ${}^Li_{E^+/F,\xi}$ is determined and
the generic automorphic induction of $\tau$ is expected to be unique. 
In general, the notion of automorphic induction is defined for all  automorphic representations and all irreducible admissible smooth representations of $\GU^\circ_{K/F}(n,F_\nu)$ over local fields. 
Here since we only consider the automorphic inductions between generic cuspidal representations, $i_{E^+/F,\xi}(\tau_\nu)$  over all local places is expected to match local Langlands functorial lift of $\tau_\nu$ via ${}^Li_{E^+/F,\xi}$. That is, $\pi=i_{E^+/F,\xi}(\tau)$ is a strong Langlands functorial lifting of $\tau$.
We leave the discussion in the future.  
\end{rmk}

Before proving Theorem \eqref{thm:automorphic-induction}, let us prove the following identity first to obtain a formula of the Asai $L$-functions of the automorphic induction.
\begin{lm}\label{lm:Asai}
Let $K/F$, $E^+/F$ and $L/F$ as above and $E^+\not\simeq E$.
Suppose that $\tau_K$ is a cuspidal automorphic representation of $\GL_n(\BA_K)$ and $i_{E^+/F}(\tau_K)$ is its automorphic induction on $\GL_{2n}(\BA_E)$.

Then 
\begin{align}
&L^S(s,i_{E^+/F}(\tau_K),Asai_{E/F}\otimes\delta_{E/F}^m)\label{eq:lemma}\\
=&L^{S_{E^+}}(s,\tau_K,Asai_{K/E^+}\otimes\delta_{K/E^+}^m)L^{S_{L}}(s,\tau_K,Asai_{K/L}\otimes\delta_{K/L}^m),\nonumber 
\end{align}
where $m$ is an integer, $S$ is a finite set of places of $F$ including of all ramified places and archimedean places
and $S_E$ (resp. $S_L$) is the set of places of $E$ (resp. $L$) lying over the places in $S$. 
\end{lm} 
\begin{proof}
To prove this lemma, it suffices to show that
\begin{equation}\label{eq:lemma-proof}
L(s,i(\tau_{K,\nu}),Aasi\otimes\delta_{E/F}^m)
=L(s,\tau_{K,\nu},Aasi\otimes\delta_{K/E^+}^m)L(s,\tau_{K,\nu},Aasi\otimes\delta_{K/L}^m),
\end{equation} 
over all unramified places $\nu$ of $F$, that is, $\nu\notin S$.


Assume that $\nu$ splits as $v_E w_E$ in $E$, and $v_E$ (resp. $w_E$) splits as $v_1v_2$ (resp. $w_1w_2$) in $K$.
In this case, $\nu$ splits as $v_{E^+}w_{E^+}$ (resp. $v_Lw_L$) in $E^+$ (resp. $L$).
Since $\Gal(E/F)$ acts transitively on $\{v_E, w_E\}$,
both $\vartheta_{E^+}$ and $\vartheta_{L}$ act transitively on $\{\{v_1,v_2\},\{w_1,w_2\}\}$.
Then, without loss of generality, we may assume that $v_{E^+}=v_1w_1$, $w_{E^+}=v_2w_2$,
$v_{L}=v_1w_2$ and $w_L=v_2w_1$.

Let $\tau_{K,\nu}=\tau_{v_1}\otimes\tau_{v_2}\otimes\tau_{w_1}\otimes\tau_{w_2}$ where $\tau_{v_i}$ and $\tau_{w_i}$ are unramified representation of $\GL_{n}(K_{v_i})$ and $\GL_n(K_{w_i})$ respectively, and $K_{v_i}\cong K_{w_i}\cong F_\nu$.

By $v_{E^+}=v_1w_1$ and $w_{E^+}=v_2w_2$, we have $\tau_{K,v_{E^+}}=\tau_{v_1}\otimes\tau_{w_1}$ and $\tau_{K,v_{E^+}}=\tau_{v_2}\otimes\tau_{w_2}$, and 
\begin{equation}\label{eq:split-split-E+}
\begin{array}{l}
L(s,\tau_{K,v_{E^+}},Aasi\otimes\delta_{K/E^+}^m)=
L(s,\tau_{v_1}\otimes \tau_{w_1})\\
L(s,\tau_{K,w_{E^+}},Aasi\otimes\delta_{K/E^+}^m)=L(s,\tau_{v_2}\otimes \tau_{w_2}).
\end{array}
\end{equation}
Similarly, we have
\begin{equation}\label{eq:split-split-L}
\begin{array}{l}
L(s,\tau_{K,v_{L}},Aasi\otimes\delta_{K/L}^m)=L(s,\tau_{v_1}\otimes\tau_{w_2})\\
L(s,\tau_{K,w_{L}},Aasi\otimes\delta_{K/L}^m)=L(s,\tau_{v_2}\otimes\tau_{w_1}). 
\end{array}
\end{equation}
Note that over the split places, $\delta_{K/E^+}^m$ and $\delta_{K/L}^m$ are trivial.

By $\nu=v_{E^+}w_{E^+}=v_{L}w_L$,
\begin{align}
 &L(s,\tau_{K,\nu},Aasi\otimes\delta_{K/E'}^m) \label{eq:split-split-E'}\\
=&L(s,\tau_{K,v_{E'}},Aasi\otimes\delta_{K/E'}^m)L(s,\tau_{K,w_{E'}},Aasi\otimes\delta_{K/E'}^m), \nonumber
\end{align}
where $E'=E^+$ or $L$.
Plugging \eqref{eq:split-split-E+} and \eqref{eq:split-split-L} into \eqref{eq:split-split-E'} respectively, we get the right hand side of \eqref{eq:lemma-proof}
\begin{align}
 &L(s,\tau_{K,\nu},Aasi\otimes\delta_{K/E^+}^m)L(s,\tau_{K,\nu},Aasi\otimes\delta_{K/L}^m)\label{eq:split-right}\\
=& L(s,\tau_{v_1}\otimes\tau_{w_1})L(s,\tau_{v_1}\otimes\tau_{w_2})L(s,\tau_{v_2}\otimes\tau_{w_1})L(s,\tau_{v_2}\otimes\tau_{w_2}).\nonumber
\end{align}

On the other hand, over the place $v_E$  of $E$, the automorphic induction $i(\tau_{K,v_E})$ is an unramified representation of $\GL_{2n}(E_{v_E})$
and its Satake parameter $c(i(\tau_{K,v_E}))=c(\tau_{v_1})\oplus c(\tau_{v_2})$.
Similarly, we have $c(i(\tau_{K,w_E}))=c(\tau_{w_1})\oplus c(\tau_{w_2})$.
Then
\begin{align}
 & L(s,i(\tau_{K,\nu}),Aasi\otimes\delta_{E/F}^m) \label{eq:split-split-induction}\\
=&L(s,i(\tau_{K,v_E}),Aasi\otimes\delta_{E/F}^m)L(s,i(\tau_{K,w_E}),Aasi\otimes\delta_{E/F}^m) \nonumber\\
=&L(s,\tau_{v_1}\otimes\tau_{w_1})L(s,\tau_{v_1}\otimes\tau_{w_2})L(s,\tau_{v_2}\otimes\tau_{w_1})L(s,\tau_{v_2}\otimes\tau_{w_2}).\nonumber
\end{align}
 
Comparing \eqref{eq:split-split-induction} and   \eqref{eq:split-right}, we obtain \eqref{eq:lemma-proof} for the places $\nu$ of $F$ totally split over $K$.

Assume that $\nu$ splits as $v_E w_E$ in $E$.
Since $\Gal(K/F)$ is the Klein 4-group, $v_E$ and $w_E$ are inert in $K$.
Since $\vartheta_{E^+}$ and $\vartheta_{L}$ act transitively on $\{v_E,w_E\}$,
$\nu$ as a place of $F$ is inert in $E^+$ and $L$.
Moreover, as a place of $E^+$ and $L$, $\nu$ splits as $v_E w_E$ in $K$.

Let $\tau_{K,\nu}=\tau_{K,v_E}\otimes\tau_{K,w_E}$ where $\tau_{K,v_E}$ and $\tau_{K,w_E}$ are unramified representation of $\GL_n(K_{v_E})$ and $\GL_n(K_{w_E})$ respectively.
Since $\nu$ is split in $K$, we have
$$
L(s,\tau_{K,\nu},Aasi\otimes\delta_{K/E^+}^m)=L(s,\tau_{K,\nu},Aasi\otimes\delta_{K/L}^m) 
=L(s,\tau_{K,v_{E}}\otimes\tau_{K,w_E}).
$$
Thus
\begin{equation}\label{eq:split-inert-E'-L}
L(s,\tau_{K,\nu},Aasi\otimes\delta_{K/E^+}^m)L(s,\tau_{K,\nu},Aasi\otimes\delta_{K/L}^m)
=L(s,\tau_{K,v_{E}}\otimes\tau_{K,w_E})^2.
\end{equation}

Since the place $\nu$ of $F$ splits as $v_E w_E$ in $E$, the automorphic induction $i(\tau_{K,\nu})$ splits as $i(\tau_{K,v_E})\otimes i(\tau_{K,w_E})$ and we have
\begin{equation}\label{eq:induction-split}
L(s,i(\tau_{K,\nu}),Aasi\otimes\delta_{E/F}^m)=L(s,i(\tau_{K,v_E})\otimes i(\tau_{K,w_E})).
\end{equation}
Since $v_E$ and $w_E$ are inert in $K$, the Satake parameter $c(i(\tau_{K,v_E}))$ equals $\sqrt{c_{v_E}(\tau_{K,v_E})}\oplus -\sqrt{c_{v_E}(\tau_{K,v_E})}$, similarly over the place $w_E$.
Then
\begin{equation}\label{eq:induction-split-tensor}
L(s,i(\tau_{K,v_E})\otimes i(\tau_{K,w_E}))=L(s,\tau_{K,v_{E}}\otimes\tau_{K,w_E})^2.	
\end{equation}

Plugging \eqref{eq:induction-split-tensor} into \eqref{eq:induction-split} and then comparing \eqref{eq:induction-split} and  \eqref{eq:split-inert-E'-L},
we conclude \eqref{eq:lemma-proof} for the case $\nu$ split in $E$ but not totally split in $K$.

Assume that $\nu$ is inert in $E$ and then $\nu$ is split in $K$, written as $\nu=vw$.
Now, we have the place $\nu$ of $F$ is split either in $E^+$ or in $L$. 
Due to symmetry, it is enough to consider the case that $\nu$ is split in $E^+$.
It follows that $\nu$ splits as $vw$ in $E^+$ and  is inert in $L$.

Let $\tau_{K,\nu}=\tau_{K,v}\otimes\tau_{K,w}$ where $\tau_{K,v}$ and $\tau_{K,w}$ are unramified representation of $\GL_n(K_{v})$ and $\GL_n(K_{w})$ respectively.
Remark that $K_{v}$ is isomorphic to $E_\nu$, which is the unramified quadratic extension of $F_\nu$.
Since the place $\nu$ of $L$ splits as $vw$ in $K$, we have
\begin{equation}\label{eq:inert-split-L}
L(s,\tau_{K,\nu},Aasi\otimes\delta_{K/L}^m)=L(s,\tau_{K,v}\otimes \tau_{K,w}).	
\end{equation}
Since $v$ and $w$ are two places of $E^+$, $K_{v}$ and $K_w$ are the unramified quadratic extensions of $E^+_v$ and $E^+_w$ respectively.
Thus
\begin{equation}\label{eq:inert-split-E+}
L(s,\tau_{K,\nu},Aasi\otimes\delta_{K/E^+}^m)=L(s,\tau_{K,v},Aasi\otimes\delta_{K/E^+}^m)
L(s,\tau_{K,w},Aasi\otimes\delta_{K/E^+}^m).	
\end{equation}
Plugging \eqref{eq:inert-split-E+} and \eqref{eq:inert-split-L} into the right hand side of \eqref{eq:lemma-proof}, we have
\begin{align}
 &L(s,\tau_{K,\nu},Aasi\otimes\delta_{K/E^+}^m) L(s,\tau_{K,\nu},Aasi\otimes\delta_{K/L}^m)\label{eq:inert-split-right}\\
=&L(s,\tau_{K,v},Aasi\otimes\delta_{K/E^+}^m)L(s,\tau_{K,w},Aasi\otimes\delta_{K/E^+}^m)
L(s,\tau_{K,v}\otimes \tau_{K,w})\nonumber.	
\end{align}

On the other hand, since  $\nu$ as a local place of $E$ is split in $K$, the automorphic induction $i(\tau_{K,\nu})$ is an unramified representation of $\GL_{2n}(E_\nu)$
and its Satake parameter $c(i(\tau_{K,\nu}))=c(\tau_{K,v})\oplus c(\tau_{K,w})$.
Hence
\begin{align}
&L(s,i(\tau_{K,\nu}),Aasi\otimes\delta_{E/F}^m) \label{eq:inert-split-E}\\
=&L(s,\tau_{K,v},Asai\otimes\delta_{E/F}^m)L(s,\tau_{K,w},Asai\otimes\delta_{E/F}^m)L(s,\tau_{K,v}\otimes \tau_{K,w}).\nonumber
\end{align}

Comparing \eqref{eq:inert-split-right} and  \eqref{eq:inert-split-E}, one obtains \eqref{eq:lemma-proof} for the case $\nu$ inert,
which completes the proof of this lemma.
\end{proof}

Finally, let us complete the proof of Theorem \ref{thm:automorphic-induction}.
\begin{proof}[Proof of Theorem \ref{thm:automorphic-induction}]
First, we claim that there exists a base change $b_{E/F,\xi}(\tau)$ for $\tau$,
which is an automorphic representation of $\GL_n(\BA_K)\times \GL_1(\BA_E)$.
It corresponds the $L$-homomorphism given by the top-horizontal arrow in Diagram \eqref{diagram}.
In fact, when $\xi$ is trivial, the standard base change has been verified in Lemma \ref{lm:Base-Change} and \cite[Theorem 1.1]{KK05}. 
Here we apply the same arguments with Lemma \ref{lm:Base-Change} by combing with the endoscopic classification of unitary groups in Theorem \ref{thm:Mok}.

Referring to \cite[Theorem 5.3]{KK08}, the restriction of $\tau$ to $\RU_{K/E^+}(n,\BA_{E^+})$ is a direct sum of cuspidal automorphic representations of $\RU_{K/E^+}(n,\BA_{E^+})$.
Let $\tau_0$ be a generic constituent of the restriction of $\tau$.
Similar to the proof of Lemma \ref{lm:Base-Change}, $b_{E/F,\xi}(\tau)=b_{E/F,\xi}(\tau_{0})\otimes\eta$, which is an automorphic representation of $\GL_{n}(\BA_K)\times\GL_1(\BA_E)$.
Here $b_{E/F,\xi}(\tau_{0})$ is a base change of $\tau_0$ from $\RU_{K/E^+}(n,\BA_{E^+})$ to $\GL_{n}(\BA_K)$ and $\eta$ is an idele character of $\GL_1(\BA_E)$.

By Theorem \ref{thm:Mok}, we have  $b_{E/F,\xi}(\tau_{0})$ is an isobaric sum of form
\[
\tau_K:=b_{E/F,\xi}(\tau_{0})=\tau_{K,1}\boxplus\tau_{K,2}\boxplus\cdots\boxplus \tau_{K,r},	
\]
where $\tau_{K,i}$ is an irreducible cuspidal automorphic representation of $\GL_{n_i}(\BA_K)$ and $L(s,\tau_{K,i}, Asai_{K/E^+}\otimes\delta^{2n-1}_{K/E^+})$ has a pole at $s=1$.
And $\omega_{\tau_K}\vert_{\BA^\times_E}=\frac{\eta^{\iota_{E/F}}}{\eta}\delta^n_{K/E}$ (see \cite[{Theorem 1.1}]{Shin} for instance),
where $\eta^{\iota_{E/F}}=\eta\circ\iota_{E/F}$ is the Galois twist.  

Second, we consider the automorphic induction $i_{E^+/F}(\tau_K)$,
which is an automorphic representation of $\GL_{2n}(\BA_E)$. 
It corresponds to the $L$-homomorphism in the right-vertical arrow of Diagram \eqref{diagram}. 
Since $L(s,\tau_{K,i}, Asai_{K/E^+}\otimes\delta_{K/E^+})$  has a pole at $s=1$ for each $i$,
by Lemma \eqref{eq:lemma},  $L^S(s,i_{E^+/F}(\tau_{K,i}),Asai_{E/F}\otimes\delta_{E/F})$ has a pole at $s=1$.
Note that 
\[
i_{E^+/F}(\tau_{K})=i_{E^+/F}(\tau_{K,1})\boxplus i_{E^+/F}(\tau_{K,2})\boxplus\cdots\boxplus i_{E^+/F}(\tau_{K,r}).
\]
Also, since $i_{E^+/F}(\tau_K)$ is the automorphic induction of $\tau_K$, we have
$\omega=\omega_{\tau_K}\vert_{\BA^\times_E}\otimes \delta^n_{K/E}$,
 where $\omega$ is the central character of $i_{E^+/F}(\tau_K)$.
By $\omega_{\tau_K}\vert_{\BA^\times_E}=\frac{\eta^{\iota_{E/F}}}{\eta}\delta^n_{K/E}$, then $\omega=\frac{\eta^{\iota_{E/F}}}{\eta}$.

Finally, 
by \cite[Theorem 1.1]{KK05} and Theorem \ref{thm:Mok}, 
there exists a cuspidal generic automorphic representation $\pi_0$ of $\RU_{2n}(\BA_F)$  such that $b_{E/F}(\pi_0)=i_{E^+/F}(\tau_K)$ is the strong base change of $\pi_0$.
By $\omega=\frac{\eta^{\iota_{E/F}}}{\eta}$,
referring to Proposition 5.5 in \cite{KK08}, 
there exists a cuspidal generic automorphic representation $\pi$ of $\GU_{2n}(\BA_F)$ with central character $\omega_\pi=\eta^{\iota_{E/F}}$ such that $b_{E/F}(\pi)=i_{E^+/F}(\tau_K)\otimes \eta$.
Since the diagram is commutative, $\pi$ is the desired cuspidal automorphic representation of $\GU_{2n}(\BA_F)$ . 


\end{proof}

\subsection{Main conjecture}
In this section, analogous to \cite{GR00} and \cite[Theorem 2]{Y}, we conjecture a criterion on the existence of poles of   $L^S(s,\pi,\e{3}\otimes\chi)$ by using the automorphic induction defined in Theorem \ref{thm:automorphic-induction}.
In addition, we verify the conjecture for certain cuspidal automorphic representations from endoscopic lifting. 

\begin{conj}\label{conj:poles}
Suppose that $\pi$ is an irreducible cuspidal automorphic representation of $\GU_6(\BA)$ in the tempered discrete  spectrum defined in Conjecture \ref{conj:endoscopy-GU}. 

Then $L^S(s,\pi,\e{3}\otimes\chi)$ has a pole at $s=1$ if and only if $\pi$ is an automorphic induction $i_{E^+/F,\xi}(\tau)$ from an irreducible cuspidal automorphic representation $\tau$ of $\GU^\circ_{K/E^+}(3,\BA_F)$ in the tempered discrete  spectrum and $L(s,i_F(\omega_\tau)\otimes\chi)$ has a pole at $s=1$,
where $K=E\otimes_F E^+$ and $E^+/F$ is the quadratic extension associate to the non-trivial quadratic character $\omega_\pi\chi^2$,  and $\omega_\tau$ is the central character of $\tau$.
\end{conj}

Here $i_F$ is an automorphic induction from the central character $\omega_\tau$ of $\GU^\circ_{K/E^+}(1,\BA_F)$ to an automorphic representation of $\GL_2(\BA_F)$,
which is defined by 
the corresponding   $L$-homomorphism ${}^Li_F$:  
${}^Li_F((1,1,1)\rtimes w)=I_2$ for $w\in W_K$, and
\begin{equation}\label{eq:induction-ordinary}
{}^Li_F((g,h,a)\rtimes 1)= 
a\left(\begin{smallmatrix}
g&0\\ 0&  h	
\end{smallmatrix}\right),\quad 
{}^Li_F((1,1,1)\rtimes \vartheta_{E})={}^Li_F((1,1,1)\rtimes \vartheta_{E^+})
=\left(\begin{smallmatrix}
0&1\\ 1& 0	
\end{smallmatrix}\right).
\end{equation}

Remark that similar to \cite{GR00} and \cite{Y}, one can replace the condition on $L(s,i_F(\omega_\tau)\otimes\chi)$ by an equivalent condition on $\omega_\tau$,
which is relative wordy. 
So we choose the current version.

To justify Conjecture \eqref{conj:poles}, we may decompose the twisted exterior cube $L$-function of $i_{E^+/F,\xi}(\tau)$ as follows
\begin{equation}\label{eq:L-functions}
L(s,i_{E^+/F,\xi}(\tau),\e{3}\otimes\chi)
=L(s,i_F(\omega_\tau)\otimes\chi)
L(s,\e{2}\tau\times\tau^{\vartheta_{E^+/F}}\otimes\chi)	
\end{equation}
where  $\tau^{\vartheta_{E^+/F}}=\tau\circ \vartheta_{E^+/F}$ is the Galois twist. 

Next, let us legitimate the decomposition \eqref{eq:L-functions} 
by explicating the $L$-embedding $\e{3}\circ {}^L i_{E^+/F,\xi}$ of ${}^L\GU^\circ_{K/E^+}(3)$ into $\GL(\e{3}\BC^6)$.

First, denote 
\begin{equation}\label{eq:W-basis}
W={\rm Span}_\BC\{e_1\wedge e_2\wedge e_3,e_4\wedge e_5\wedge e_6\}.	
\end{equation}
Under the symplectic form $q(\cdot,\cdot)$ on $\e{3}\BC^6$ defined in \eqref{eq:q},
denote $W^\perp$ to be the orthogonal complementary of $W$, which equals 
\begin{equation*}\label{eq:basis-perp}
W^\perp={\rm Span}_\BC~\CB_3\smallsetminus 
\{e_1\wedge e_2\wedge e_3,e_4\wedge e_5\wedge e_6\},
\end{equation*}
where $\CB_3$ is the basis defined in \eqref{eq:basis}.
Thus the representation $\e{3}\circ {}^Li_{E^+/F,\xi}$ is decomposed as the direct sum of two irreducible representations
\begin{equation}\label{eq:decomp}
\e{3}\circ {}^Li_{E^+/F,\xi}=\rho_W\oplus \rho_{W^\perp}
\end{equation} 
where  $(\rho_W, W)$ and $(\rho_{W^\perp}, W^\perp)$ are irreducible representations of ${}^L\GU^\circ_{K/E^+}(3)$.
Note that these two representations are independent with the choice of $\xi$.

Second, we give explicit descriptions on the representations $(\rho_W, W)$ and $(\rho_{W^\perp}, W^\perp)$, respectively.
Under the choice of basis in \eqref{eq:W-basis}, $\rho_W$ maps
\begin{align*}
 &  (g,h,a)\rtimes 1\mapsto 
a\left(\begin{smallmatrix}
\det(g)&0\\ 0&  \det(h)	
\end{smallmatrix}\right) &&
 (I_3,I_3,1)\rtimes w\mapsto I_2 \text{ (for $w\in W_K$)}\\
& (I_3,I_3,1)\rtimes \vartheta_{E}\mapsto \left(\begin{smallmatrix}
0&1\\ 1& 0	
\end{smallmatrix}\right) &&
(I_3,I_3,1)\rtimes \vartheta_{E^+}\mapsto \left(\begin{smallmatrix}
0&1\\ 1& 0	
\end{smallmatrix}\right),
\end{align*}
which is the $L$-homomorphism ${}^Li_F$ defined in \eqref{eq:induction-ordinary}.

To describe the representation $(\rho_{W^\perp}, W^\perp)$, let us start with the
subgroup $\CH=\widehat{\GU}^\circ_{K/E^+}(3)\rtimes \langle\vartheta_{E^+/F}\rangle$ of ${}^L\GU^\circ_{K/E^+}(3)$ of index at most 2.
We extend the embedding of $\widehat{\GU}^\circ_{K/E^+}(3)$ into $\GL_6(\BC)$ via the block-diagonal matrices, 
to an irreducible representation of $\CH$, denoted by $\St_{\CH}$, 
namely,
$$
\St_{\CH}\colon (g,h,a)\mapsto a \begin{pmatrix}
g&0\\ 0&  h	
\end{pmatrix},\quad
 (I_3,I_3,1)\rtimes\vartheta_{E^+/F}\mapsto \begin{pmatrix}
0&I_3\\ I_3& 0	
\end{pmatrix}.
$$
Restricting into the subgroup $\widehat{\GU}^\circ_{K/E^+}(3)$ of $\CH$, 
$\St_\CH$ is semisimple and decomposed by $V_1\oplus V_2$,
where $V_1$ and $V_2$ are isomorphic to the standard representations of $\GL_3(\BC)$, corresponding to the above block-diagonal emedding.

Define $\e{2}_{\CH}$ to be the exterior representations of $\CH$ by, for $(g,h,a)\in \GL_3(\BC)\times\GL_3(\BC)\times \GL_1(\BC)$
$$
\e{2}_{\CH}\colon (g,h,a)\mapsto a \begin{pmatrix}
\e{2}g&0\\ 0& \e{2}h	
\end{pmatrix},\quad
 (I_3,I_3,1)\rtimes\vartheta_{E^+/F}\mapsto \begin{pmatrix}
0&I_3\\ I_3& 0	
\end{pmatrix},
$$
where $\e{2}$ is the exterior cube representation of $V_{i}$, i.e., for $g\in\GL_3(\BC)$, $\e{2}g=\det(g)\cdot {}^tg^{-1}$.
Then $\e{2}_{\CH}\otimes\St_{\CH}$ is the tensor representation of the above two of $\CH$, which is of dimension 36 and 
a direct sum of two irreducible representations of $\CH$.
Indeed,
$$
(\e{2}V_1\oplus \e{2}V_2)\otimes (V_1\oplus V_2)
=(\e{2}V_1\otimes V_2\oplus \e{2}V_2\otimes V_1)
\oplus (\e{2}V_1\otimes V_1\oplus \e{2}V_2\otimes V_2),
$$
which are denoted by $\e{2}_{\CH}\otimes\St^{\vartheta_{E^+/F}}_{\CH}$ and $\e{2}_{\CH}\otimes\St_{\CH}$ respectively. 
Note that both of these two representations are of dimension 18.

It is easy to check that the representation $\e{2}_{\CH}\otimes\St^{\vartheta_{E^+/F}}_{\CH}$ can be extended to  a representation of ${}^L\GU^\circ_{K/E^+}(3)$.
We  denote it by $\e{2}\otimes\St^{\vartheta_{E^+/F}}$ if no confusion is caused.
Then it is isomorphic to the representation $\rho_{W^\perp}$ in \eqref{eq:decomp},
that is
$$\rho_{W^\perp}=\e{2}\otimes\St^{\vartheta_{E^+/F}}.$$
In fact, 
\begin{align*}
 &\rho_{W^\perp}\colon	(I_3,I_3,1)\rtimes w\mapsto I_{18}\qquad \text{ (for $w\in W_K$)} \\
 & \rho_{W^\perp}((I_3,I_3,1)\rtimes\vartheta_{E/F})=
 \rho_{W^\perp}((I_3,I_3,1)\rtimes\vartheta_{E^+/F}).
\end{align*} 

\subsubsection{Example}\label{ex}
In this example, we will prove that Conjecture \ref{conj:poles} holds for some cuspidal representations arising from certain endoscopic lifting.

Let $\pi$ be an irreducible generic cuspidal automorphic representation of $\GU_6(\BA)$,
whose restriction to $\RU_6(\BA)$ is a direct sum of cuspidal automorphic representations of $\RU_6(\BA)$.
Let $\pi_0$ be a generic cuspidal constitute of $\pi\vert_{\RU_6(\BA)}$
and  $b_{E/F}(\pi_0)$ be the strong base change lift of $\pi_0$ (defined in \cite{KK05} for instance). 
Then $b_{E/F}(\pi)=b_{E/F}(\pi_0)\otimes\bar{\omega}_\pi$ is the stable base change lift of $\pi$,
where $\bar{\omega}_\pi:=\omega_\pi\circ\iota_{E/F}$. 

Assume that $b_{E/F}(\pi_0)$ is of the form 
\begin{equation}\label{eq:base-change-pi0}
b_{E/F}(\pi_0)=\sig_1\boxplus\sig_2\boxplus\cdots\boxplus\sig_{k},	
\end{equation}
where all $\sig_i$  are (unitary) cuspidal representations of $\GL_{n_i}(\BA_E)$ such that $L(s,\sig_i,Asai\otimes\delta_{E/F})$ has a simple pole at $s = 1$ and $\sum_{i=1}^k n_i=6$. 
Moreover, 
$$
\omega_{\sig_i}\vert_{\BA^\times}=\begin{cases}
	1  &\text{ if $n_i$ is even}\\
	\delta_{E/F}  &\text{ if $n_i$ is odd,}
\end{cases}
$$
and $\prod_{i=1}^{k}\omega_{\sig_i}=\omega_\pi/\bar{\omega}_\pi$.
Without loss of generality, we may assume 
$$
n_1\geq n_2\geq\cdots \geq n_r>0.
$$

Next, let us verify Conjecture \ref{conj:poles} for the cuspidal representation  $\pi$ 
that is an endoscopic lifting from 
a generic cuspidal automorphic representation $\tau$ of $\RG(\RU_3\times\RU_3)$.

\begin{pro} \label{pro:n=3}
Following the above notation, suppose that  
$n_1=n_2=3$ in \eqref{eq:base-change-pi0} and $\omega_\pi\chi^2\vert_{\BA^\times}=\delta_{E/F}$.

Then $L^S(s,\pi,\e{3}\otimes\chi)$ has a pole at $s=1$ if and only if   $L(s,i_F(\omega_\tau)\otimes\chi)$ has a pole at $s=1$. 
\end{pro}

\begin{proof}
We prove Conjecture \ref{conj:poles} by studing the exterior cube $L$-function of its base change $b_{E^+/F}(\pi)$.
By the assumption, we have $b_{E/F}(\pi_0)=\sig_1\boxplus \sig_2$,
where $\sig_1$ and $\sig_2$ are cuspidal automorphic representations of $\GL_3(\BA_E)$.
In addition, $\omega_{\sig_1}\omega_{\sig_2}=\omega_\pi/\bar{\omega}_\pi$.
By $\omega_\pi\vert_{\BA^\times}\cdot \chi^2=\delta_{E^+/F}$, we have
$$
\omega_{\sig_1}\omega_{\sig_2}(\bar{\omega}_\pi\chi_E)^2
=(\omega_\pi\vert_{\BA}\cdot \chi^2)\circ N_{E/F}=\delta_{E^+/F}\circ N_{E/F}.
$$
More precisely,
\begin{equation} \label{eq:central-two-sigs}
\omega_{\sig_1}\omega_{\sig_2}\cdot(\bar{\omega}_\pi\chi_E)^2=\begin{cases}
	1 &\text{ if }E^+\simeq E\\
	\delta_{K/E} &\text{ if }E^+\not\simeq E.
\end{cases}	
\end{equation}

Note that 
\begin{equation}\label{eq:L-BC-GU-E+=E}
L^S(s,b_{E/F}(\pi_0),\e{3}\otimes\bar{\omega}_\pi\chi_E)=L^S(s,\pi,\e{3}\otimes\chi)L^S(s,\pi,\e{3}\otimes\chi\delta_{E/F}).	
\end{equation}
Following  \eqref{eq:L-BC-GU-E+=E}, 
we may study the poles of $L^S(s,b_{E/F}(\pi_0),\e{3}\otimes\bar{\omega}_\pi\chi_E)$,
by writing it in terms of $L$-functions of $\sig_1$ and $\sig_2$.
Since $\sig_i$ for $i=1,2$ is cuspidal automorphic representation of $\GL_3(\BA_E)$, 
we have $\e{2}\sig_i=\omega_{\sig_i}\otimes\sig_i^\vee$ and
\begin{align*}
 L^S(s,b_{E/F}(\pi_0),\e{3}\otimes\bar{\omega}_\pi\chi_E)
=&L^S(s,\omega_{\sig_1}\bar{\omega}_\pi\chi_E)L^S(s, (\bar{\omega}_\pi\chi_E\omega_{\sig_2}\otimes\sig_1)\times\sig^\vee_2)\\
&L^S(s,\omega_{\sig_2}\bar{\omega}_\pi\chi_E)
L^S(s, (\bar{\omega}_\pi\chi_E\omega_{\sig_1}\otimes\sig_2)\times\sig^\vee_1).
\end{align*}

First, let us compare the central characters of $\bar{\omega}_\pi\chi_E\omega_{\sig_{i}}\otimes\sig_{3-i}$ and $\sig_{i}$ for $i=1,2$, which are $(\bar{\omega}_\pi\chi_E\omega_{\sig_{i}})^3\omega_{\sig_{3-i}}$ and $\omega_{\sig_{i}}$ respectively.
we have $(\bar{\omega}_\pi\chi_E\omega_{\sig_{i}})^3\omega_{\sig_{3-i}}=\omega_{\sig_{i}} $ if and only if 
\begin{equation} \label{eq:Base-change-type-3-3}
\omega_{\sig_i}\bar{\omega}_\pi\chi_{E}= \begin{cases}
	1 &\text{ if }E^+\simeq E\\
	\delta_{K/E} &\text{ if }E^+\not\simeq E 
\end{cases}  	
\end{equation}
by \eqref{eq:central-two-sigs}.

In our case, we have $E^+\simeq E$.
We claim that $L^S(s,b_{E/F}(\pi_0),\e{3}\otimes\bar{\omega}_\pi\chi_E)$ has a pole at $s=1$ if and only if $\omega_{\sig_1}\bar{\omega}_\pi\chi_{E}=\omega_{\sig_2}\bar{\omega}_\pi\chi_{E}=1$.
Moreover, if $s=1$ is a pole, then it is a double pole. 
Remark that $\omega_{\sig_1}\bar{\omega}_\pi\chi_{E}=1$ if and only if $\omega_{\sig_2}\bar{\omega}_\pi\chi_{E}=1$ by \eqref{eq:central-two-sigs}.

If $\omega_{\sig_1}\bar{\omega}_\pi\chi_{E}=\omega_{\sig_2}\bar{\omega}_\pi\chi_{E}=1$, then $L^S(s,\omega_{\sig_i}\bar{\omega}_\pi\chi_{E})$ has a pole at $s=1$ and $L^S(s,b_{E/F}(\pi_0),\e{3}\otimes\bar{\omega}_\pi\chi_E)$ has a pole at $s=1$ of order 2 at least. 
Following from \eqref{eq:L-BC-GU-E+=E}, the order of the pole at $s=1$ is at most 2.
Thus, the order of the pole is exactly 2.
On the other hand, if $L^S(s,b_{E/F}(\pi_0),\e{3}\otimes\bar{\omega}_\pi\chi_E)$ has a pole at $s=1$, then either $L^S(s,\omega_{\sig_i}\bar{\omega}_\pi\chi_{E})$ or $L^S(s, (\bar{\omega}_\pi\chi_E\omega_{\sig_i}\otimes\sig_1)\times\sig^\vee_{3-i})$ has a pole for some $i$.
If  $L^S(s,\omega_{\sig_i}\bar{\omega}_\pi\chi_{E})$ has a pole, then $\omega_{\sig_i}\bar{\omega}_\pi\chi_{E}=1$.
Even if $L^S(s, (\bar{\omega}_\pi\chi_E\omega_{\sig_i}\otimes\sig_i)\times\sig^\vee_{3-i})$ admits a pole, by \eqref{eq:Base-change-type-3-3} we still have $\omega_{\sig_i}\bar{\omega}_\pi\chi_{E}=1$.
In all cases, it follows that $\omega_{\sig_1}\bar{\omega}_\pi\chi_{E}=\omega_{\sig_2}\bar{\omega}_\pi\chi_{E}=1$.
That is, $L^S(s,i_F(\omega_\tau)\otimes\chi)=L^S(s,\triv)L^S(s,\delta_{E/F})$ has a pole at $s=1$.

This completes the proof of this proposition.
\end{proof}

\begin{rmk}
In Proposition \ref{pro:n=3}, suppose that $\omega_\pi\chi^2\vert_{\BA^\times}=\delta_{E/F}$, i.e., $E^+\not\simeq E$.
Following a similar argument of Proposition \ref{pro:n=3}, one may have $L^S(s,i_F(\omega_\tau)\otimes\chi)$ is holomorphic at $s=1$.
For this type of representations, $L^S(s,\pi,\e{3}\otimes\chi)$ is holomorphic at $s=1$.
To establish Conjecture \ref{conj:poles} for $E^+\not\simeq E$, 
we need more refined results on the automorphic induction to show that such $\pi$ is not an automorphic induction $i_{E^+/F,\xi}(\tau)$ for any $\tau$.
We will leave the discuss in the future.
\end{rmk}

\end{document}